\documentclass[a4paper,final,fleqn]{article}

\RequirePackage [english] {babel}
\RequirePackage [utf8]    {inputenc}
\RequirePackage [T1]      {fontenc}

\RequirePackage {amsmath,mathtools}

\RequirePackage {lmodern}
\RequirePackage {amssymb}
\RequirePackage [mathscr] {eucal}

\RequirePackage {bbding}

\RequirePackage {ifthen,xspace,enumitem}

\RequirePackage {fixltx2e,titlesec,fancyhdr}

\RequirePackage [usenames,dvipsnames] {xcolor}

\RequirePackage [thmmarks,amsmath,amsthm]
			    {ntheorem}

\RequirePackage[
final,
hyperfootnotes=true,
hyperindex=true,
breaklinks=true,
colorlinks=true,
linkcolor=linkred,
citecolor=gray,
filecolor=magenta,
urlcolor=linkred,
anchorcolor=blue,
menucolor=red,
plainpages=false,
unicode=true,
pdfproducer={pdflatex},
bookmarks=true,
bookmarksnumbered=true,
bookmarksopen=true,
bookmarksopenlevel=2,
pdfview=FitH, 
pdfstartview=FitH, 
pdfstartpage={1},
pdfpagelayout=OneColumn
]{hyperref}

\definecolor {darkblue}  {rgb} {0.09,0.20,0.43}
\definecolor {darkgray}  {rgb} {0.39,0.39,0.40}
\definecolor {darkgreen} {rgb} {0,0.5,0}
\definecolor {purple}    {rgb} {0.55,.16,0.44}
\definecolor {linkred}   {rgb} {0.545,0.275,0.29}

\definecolor{accent}{rgb}{0.25,0.25,0.25}

\titleformat {\section} {\color{accent}\normalfont\bfseries\large}{\thesection}{1em}{}
\titleformat {\subsection} {\color{accent}\normalfont\bfseries}{\thesubsection}{1em}{}

\let\oldtitle\title
\def\title #1 {\oldtitle{\bfseries\color{accent}#1}}

	{\begin{enumerate}[label=(\roman{*})]}{\end{enumerate}}
	{\begin{enumerate}[label={\rm (\roman{*})}]}{\end{enumerate}}

\newcommand{\accentheaderfont}%
	{\normalfont\normalsize\bfseries\color{accent}}
\newcommand{\prooffont}%
	{\normalfont\normalsize\itshape}

\newtheoremstyle{accentheader}%
	{\item[\hskip\labelsep \accentheaderfont{##1 ##2}]}%
	{\item[\hskip\labelsep \accentheaderfont{##1 ##2 (##3)}]}%

\newtheoremstyle{nonumberaccentheader}%
	{\item[\hskip\labelsep \accentheaderfont{##1}]}%
	{\item[\hskip\labelsep \accentheaderfont{##1 (##2)}]}%

\newtheoremstyle{blackheader}
 	{\item[\hskip\labelsep \prooffont{##1}]}%
 	{\item[\hskip\labelsep \prooffont{##3}]}

\newtheoremstyle{nonumberblackheader}%
	{\item[\hskip\labelsep \prooffont{##1}]}%
	{\item[\hskip\labelsep \prooffont{##1 (##2)}]}%

\newtheoremstyle{blackheadernumbered}
 	{\item[\hskip\labelsep \prooffont{##1 ##2.}]}%
 	{\item[\hskip\labelsep \prooffont{##1 ##2 (##3).}]}

\numberwithin{dummy}{section}

\theoremstyle{accentheader}
\theorembodyfont{\itshape}
\theoremseparator{}
\theoremsymbol{}

\newtheorem{thm}{Theorem}[section]

\newtheorem{prop}[thm]{Proposition}

\newtheorem{lem}[thm]{Lemma}

\newtheorem{cor}[thm]{Corollary}

\theoremstyle{blackheader}
\theorembodyfont{\normalfont}

\renewtheorem{proof}{Proof.}

\theoremstyle{blackheadernumbered}
\theorembodyfont{\normalfont}
\newtheorem{rem}[thm]{Remark}

\theoremstyle{accentheader}
\theorembodyfont{\itshape}

\renewtheorem{prop*}{Proposition}
\renewtheorem{lem*}{Lemma}
\renewtheorem{defn*}{Definition}
\renewtheorem{cor*}{Corollary}

\theoremstyle{blackheader}
\theorembodyfont{\normalfont}

\renewtheorem{rem*}{Remark.}

\let\existsorig\exists 
\renewcommand{\exists}{\ \existsorig\ } 
\let\forallorig\forall 
\renewcommand{\forall}{\ \forallorig\ } 

\newcommand{\overbar}[1]{\mkern 1.5mu\overline{\mkern-1.5mu#1\mkern-1.5mu}\mkern 1.5mu}
\newcommand{\ob}[1]{\overbar{#1}}

\renewcommand{\atop}[2]{\genfrac{}{}{0pt}{}{#1}{#2}}

 \newcommand{\defl}{\coloneqq}%
 \newcommand{\defr}{\eqqcolon}%

\renewcommand{\le}{\leqslant}
\renewcommand{\ge}{\geqslant}

\newcommand{\lex}{\preccurlyeq}

\renewcommand{\Re}{\mathfrak{R}}

\renewcommand{\qed}{{\quad~\color{accent}\text{\small\RectangleBold}}}
\newcommand{\fish}%
	{\quad\ensuremath{\color{accent}\rtimes}}
\newcommand{\map}%
	{\quad\ensuremath{\multimap}}

\DeclarePairedDelimiter{\abs}{\lvert}{\rvert}

\DeclarePairedDelimiter{\n}{\lVert}{\rVert}
\DeclarePairedDelimiter{\lin}{\langle}{\rangle}
\DeclarePairedDelimiter{\pbr}{\lbrace}{\rbrace}
\DeclarePairedDelimiterX{\setdef}[2]{\{}{\}}{#1\,:\,#2}
\DeclarePairedDelimiter{\setd}{\{}{\}}

\DeclarePairedDelimiter{\p}{(}{)}

\newcommand{\e}{\mathrm{e}}
\newcommand{\ii}{\mathrm{i}}

\newcommand{\opento}{\hookrightarrow}%

\DeclareMathOperator{\dist}{dist}

\newcommand{\Id}{I}

\renewcommand{\C}{\mathbb{C}}%
\newcommand{\N}{\mathbb{N}}%
\newcommand{\R}{\mathbb{R}}%
\newcommand{\Z}{\mathbb{Z}}%
\newcommand{\T}{\mathbb{T}}%

\newcommand{\Ic}{\mathcal{I}}

\newcommand{\Pc}{\mathcal{P}}
\newcommand{\Qc}{\mathcal{Q}}

\newcommand{\Sc}{\mathcal{S}}
\newcommand{\Tc}{\mathcal{T}}

\newcommand{\Fs}{\mathscr{F}}

\newcommand{\Ms}{\mathscr{M}}

\newcommand{\Ws}{\mathscr{W}}

\newcommand{\al}{\alpha}
\newcommand{\bt}{\beta}
\newcommand{\gm}{\gamma}	 
     \newcommand{\Dl}{\Delta}
\newcommand{\ep}{\varepsilon} 

\newcommand{\lm}{\lambda}    \newcommand{\Lm}{\Lambda}

\newcommand{\sg}{\sigma}     

\newcommand{\om}{\omega}     \newcommand{\Om}{\Omega}
\newcommand{\ph}{\varphi}    \newcommand{\Ph}{\Phi}

\newcommand{\upd}{\mathinner{\mathrm{d}\kern.04em\!}}

\newcommand\ddd{\mathrm{d}\mkern1mu}

\newcommand{\dlm}{\upd{\lm}}

\newcommand{\dt}{\upd{t}}

\newcommand{\dx}{\upd{x}}

\usepackage[normalem]{ulem}

\usepackage[numbers]{natbib}
\renewcommand{\thesection}{\arabic{section}}

\def\@biblabel#1{#1.}

\DeclareMathAlphabet{\mathpzc}{OT1}{pzc}{m}{it}

\usepackage{showkeys}
\usepackage[textwidth=3.5cm,color=lightgray]{todonotes}

\usepackage[inner=4.5cm,outer=4.5cm,marginparwidth=3cm, marginparsep=.5cm]{geometry}

\renewcommand{\o}{\mathrm{o}}
\newcommand{\loc}{\mathrm{loc}}

\newcommand{\wto}{\rightharpoondown}

\newcommand{\Fm}{\Om}

\newcommand{\spii}[1]{\lin{#1}}
\newcommand{\w}[1]{\lin{#1}}
\newcommand{\spanx}{\mathrm{span}}
\newcommand{\wsto}{\overset{*}{\wto}}

\renewcommand{\Id}{\mathrm{Id}}

\DeclareMathOperator{\spec}{spec}
\DeclareMathOperator{\dir}{dir}
\DeclareMathOperator{\per}{per}

\DeclareMathOperator{\Null}{Null}
\DeclareMathOperator{\Iso}{Iso}
\newcommand{\Hm}{{\color{Fuchsia}H}}

\newcommand{\Wp}{\Ws}

\makeatletter
\newcommand\newsubcommand[3]{\def#1{#2\sc@sub{#3}}}
\def\sc@sub#1{\def\sc@thesub{#1}\@ifnextchar_{\sc@mergesubs}{_{\sc@thesub}}}
\def\sc@mergesubs_#1{_{\sc@thesub#1}}
\makeatother

\newcommand{\FL}{\Fs\ell}

\title{\boldmath On the wellposedness of the KdV equation on the space of pseudomeasures}
\author{Thomas Kappeler\footnote{Partially supported by the Swiss National Science Foundation
}, Jan Molnar\footnote{Partially supported by the Swiss National Science Foundation
}}
\date{\today}

\newcommand{\KdV}{KdV\xspace}
\newcommand{\spi}[1]{\lin{#1}_{\Ic}}

\begin{document}

\maketitle

\begin{abstract}
In this paper we prove a wellposedness result of the KdV equation on the space of periodic pseudo-measures, also referred to as the Fourier Lebesgue space $\mathscr{F}\ell^{\infty}(\mathbb{T},\mathbb{R})$, where $\mathscr{F}\ell^{\infty}(\mathbb{T},\mathbb{R})$ is endowed with the weak* topology. Actually, it holds on any weighted Fourier Lebesgue space $\mathscr{F}\ell^{s,\infty}(\mathbb{T},\mathbb{R})$ with $-1/2 < s \le 0$ and improves on a wellposedness result of Bourgain for small Borel measures as initial data. A key ingredient of the proof is a characterization for a distribution $q$ in the Sobolev space $H^{-1}(\mathbb{T},\mathbb{R})$ to be in $\mathscr{F}\ell^{\infty}(\mathbb{T},\mathbb{R})$ in terms of asymptotic behavior of spectral quantities of the Hill operator $-\partial_{x}^{2} + q$. In addition, wellposedness results for the KdV equation on the Wiener algebra are proved.



\end{abstract}


\section{Introduction}

In this paper we consider the initial value problem for the Korteweg-de Vries equation on the circle $\T = \R/\Z$,
\begin{equation}
  \label{kdv}
  \partial_{t}u = -\partial_{x}^{3}u + 6u\partial_{x}u,\qquad x\in\T,\quad t\in\R.
\end{equation}
Our goal is to improve the result of \citet{Bourgain:1997gg} on global wellposedness for solutions evolving in the Fourier Lebesgue space $\FL_{0}^{\infty}$ with small Borel measures as initial data. The space $\FL_{0}^{\infty}$ consists of 1-periodic distributions $q\in S'(\T,\R)$ whose Fourier coefficients $q_{k} = \spii{q,\e^{\ii k\pi x}}$, $k\in\Z$,  satisfy $(q_{k})_{k\in\Z}\in\ell^{\infty}(\Z,\R)$ and $q_{0} = 0$. Here and in the sequel we view for convenience 1-periodic distributions as 2-periodic ones and denote by $\lin{f,g}$ the $L^{2}$-inner product $\frac{1}{2}\int_{0}^{2} f(x)\ob{g(x)}\,\dx$ extended by duality to $S'(\R/2\Z,\C)\times C^{\infty}(\R/2\Z,\C)$.
 We point out that $q_{2k+1} = 0$ for any $k\in\Z$ since $q$ is a 1-periodic distribution.
 We succeed in dropping the smallness condition on the initial data and can allow for arbitrary initial data $q\in\FL_{0}^{\infty}$. In fact, our wellposedness results hold true for any of the spaces $\FL_{0}^{s,\infty}$ with $-1/2 < s \le 0$ where
\[
  \FL_{0}^{s,\infty} = \setdef{q\in S'(\T,\R)}{q_{0} = 0 \text{ and }\n{(q_{k})_{k\in\Z}}_{s,\infty} < \infty},
\]
and
\[
  \n{(q_{k})_{k\in\Z}}_{s,\infty} \defl \sup_{k\in\Z} \w{k}^{s}\abs{q_{k}},\qquad \w{\al}\defl 1 + \abs{\al}.
\]
Informally stated, our result says that for any $-1/2 < s \le 0$, the KdV equation is globally $C^{0}$-wellposed on $\FL_{0}^{s,\infty}$. To state it more precisely, we first need to recall the wellposedness results established in~\cite{Kappeler:2006fr} on the Sobolev space $H_{0}^{-1}\equiv H_{0}^{-1}(\T,\R)$. Let $-\infty \le a < b \le \infty$ be given. A continuous curve $\gm\colon (a,b)\to H_{0}^{-1}$ with $\gm(0) = q\in H_{0}^{-1}$ is called a solution of~\eqref{kdv} with initial data $q$ if and only if for any sequence of $C^{\infty}$-potentials $(q^{(m)})_{m\ge 1}$ converging to $q$ in $H_{0}^{-1}$, the corresponding sequence $(\Sc(t,q^{(m)}))_{m\ge 1}$ of solutions of~\eqref{kdv} with initial data $q^{(m)}$ converges to $\gm(t)$ in $H_{0}^{-1}$ for any $t\in (a,b)$. In~\cite{Kappeler:2006fr} it was proved that the KdV equation is globally in time $C^{0}$-wellposed meaning that for any $q\in H_{0}^{-1}$~\eqref{kdv} admits a solution $\gm\colon \R\to H_{0}^{-1}$ with initial data in the above sense and for any $T > 0$ the solution map $\Sc\colon H_{0}^{-1}\to C([-T,T],H_{0}^{-1})$ is continuous. Note that for any $-1/2 < s\le 0$, $\FL_{0}^{s,\infty}$ continuously embeds into $H_{0}^{-1}$.
On $\FL_{0}^{s,\infty}$ we denote by $\tau_{w*}$ the weak* topology $\sg(\FL_{0}^{s,\infty},\FL_{0}^{-s,1})$. We refer to Appendix~\ref{app:weak-star} for a discussion.

\begin{thm}
\label{thm:kdv-wp}
For any $q\in \FL_{0}^{s,\infty}$ with $-1/2 < s \le 0$, the solution curve $t\mapsto \Sc(t,q)$ evolves in $\FL_{0}^{s,\infty}$. It is bounded, $\sup_{t\in \R} \n{\Sc(t,q)}_{s,\infty} < \infty$ and continuous with respect to the weak* topology $\tau_{w*}$.~\fish
\end{thm}

\begin{rem}
\label{rem-airy}
It is easy to see that for generic initial data, the solution curve $t\mapsto \Sc_{\text{Airy}}(t,q)$ of the Airy equation, $\partial_{t}u = -\partial_{x}^{3}u$, is not continuous with respect to the norm topology of $\FL_{0}^{s,\infty}$. A similar result holds true for the KdV equation at least for small initial data -- see Section~\ref{s:main-proofs}.~\map
\end{rem}

We say that a subset $V\subset\FL_{0}^{s,\infty}$ is KdV-invariant if for any $q\in\FL_{0}^{s,\infty}$, $\Sc(t,q)\in V$ for any $t\in\R$.

\begin{thm}
\label{thm:invariant}
Let $V\subset \FL_{0}^{s,\infty}$ with $-1/2 < s \le 0$ be a KdV-invariant $\n{\cdot}_{s,\infty}$-norm bounded subset, that is
\[
  \Sc(t,q)\in V,\quad \forall t\in\R,\; q\in V;\qquad \sup_{q\in V}\n{q}_{s,\infty} < \infty.
\]
Then for any $T > 0$, the restriction of the solution map $\Sc$ to $V$ is weak* continuous,
\[
  \Sc\colon (V,\tau_{w*})\to C([-T,T],(V,\tau_{w*})).~\fish
\]
\end{thm}

By the same methods we also prove that the KdV equation is globally $C^{\om}$-wellposed on the Wiener algebra $\FL_{0}^{0,1}$ -- see Section~\ref{s:wiener} where such a result is proved for the weighted Fourier Lebesgue space $\FL_{0}^{N,1}$, $N\in\Z_{\ge 0}$.

\emph{Method of proof.}
Theorem~\ref{thm:kdv-wp} and Theorem~\ref{thm:invariant} are proved by the method of normal forms. We show that the restriction of the Birkhoff map $\Phi\colon H_{0}^{-1}\to \ell_{0}^{-1/2,2}$ $q\mapsto z(q) = (z_{n}(q))_{n\in\Z}$, constructed in \cite{Kappeler:2005fb}, to $\FL_{0}^{s,\infty}$ is a map with values in $\ell_{0}^{s+1/2,\infty}(\Z,\C)$, having the following properties:

\begin{thm}
\label{thm:bhf}
For any $-1/2 < s\le 0$, $\Phi\colon \FL_{0}^{s,\infty}\to \ell_{0}^{s,\infty}$ is a bijective, bounded, real analytic map between the two Banach spaces.
Near the origin, $\Phi$ is a local diffeomorphism.
When restricted to any $\n{\cdot}_{s,\infty}$-norm bounded subset $V\subset \FL_{0}^{s,\infty}$, $\Phi\colon V\to \Phi(V)$ is a homeomorphism when $V$ and $\Phi(V)$ are endowed with the weak* topologies $\sg(\FL_{0}^{s,\infty},\FL_{0}^{-s,1})$ and $\sg(\ell_{0}^{s+1/2,\infty},\ell_{0}^{-(s+1/2),1})$, respectively.
Furthermore for any $q\in \FL_{0}^{s,\infty}$, the set $\Iso(q)$ of elements $\tilde q\in H_{0}^{-1}$ so that $-\partial_{x}^{2}+q$ and $-\partial_{x}^{2}+\tilde q$ have the same periodic spectrum is a $\n{\cdot}_{s,\infty}$-norm bounded subset of $\FL_{0}^{s,\infty}$ and hence $\Phi\colon \Iso(q)\to \Phi(\Iso(q))$ is a homeomorphism when $\Iso(q)$ and $\Phi(\Iso(q))$ are endowed with the weak* topologies.~\fish
\end{thm}

\begin{rem}
Note that by~\cite{Kappeler:2005fb} for any $q\in H_{0}^{-1}$, $\Phi(\Iso(q)) = \Tc_{\Phi(q)}$ where
\begin{equation}
  \label{torus}
  \Tc_{\Phi(q)}
   = \setdef{\tilde z = (\tilde z_{k})_{k\in\Z}\in \ell_{0}^{-1/2,2}}
            {\abs{\tilde z_{k}} = \abs{\Phi(q)_{k}} \forall k\in\Z}.
\end{equation}
Furthermore, since by Theorem~\ref{thm:bhf} for any $q\in\FL_{0}^{s,\infty}$, $-1/2 < s \le 0$, $\Iso(q)$ is bounded in $\FL_{0}^{s,\infty}$, the weak* topology on $\Iso(q)$ coincides with the one induced by the norm $\n{\cdot}_{\sg,p}$ for any $-1/2 < \sg < s$, $2\le p < \infty$ with $(s-\sg)p > 1$ -- cf. Lemma~\ref{weak-star} from Appendix~\ref{app:weak-star}.~\map
\end{rem}

Key ingredient for studying the restriction of the Birkhoff map to the Banach spaces $\FL_{0}^{s,\infty}$ are pertinent asymptotic estimates of spectral quantities of the Schrödinger operator $-\partial_{x}^{2}+q$, which appear in the estimates of the Birkhoff coordinates in \cite{Kappeler:2005fb,Kappeler:CNzeErmy} -- see Section~\ref{s:spectral-theory}.
The proofs of Theorem~\ref{thm:kdv-wp} and Theorem~\ref{thm:invariant} then are obtained by studying the restriction of the solution map $\Sc$, defined in \cite{Kappeler:2006fr} on $H_{0}^{-1}$, to $\FL_{0}^{s,\infty}$.
To this end, the KdV equation is expressed in Birkhoff coordinates $z = (z_{n})_{n\in\Z}$. It takes the form
\[
  \partial_{t}z_{n} = -\ii \om_{n}z_{n},\qquad
  \partial_{t}z_{-n} = \ii \om_{n}z_{-n},\qquad n\ge 1,
\]
where $\om_{n}$, $n\ge 1$, are the KdV frequencies. For $q\in H_{0}^{1}$, these frequencies are defined in terms of the KdV Hamiltonian $\Hm(q) = \int_{0}^{1} \p*{\frac{1}{2}(\partial_{x}q)^{2} + q^{3} }\,\dx$. When viewed as a function of the Birkhoff coordinates, $\Hm$ is a real analytic function of the actions $I_{n} = z_{n}z_{-n}$, $n\ge 1$, alone and $\om_{n}$ is given by
\[
  \om_{n} = \partial_{I_{n}}\Hm.
\]
For $q\in H_{0}^{-1}$, the KdV frequencies are defined by analytic extension -- see~\cite{Kappeler:2016uj} for novel formulas allowing to derive asymptotic estimates.

\emph{Related results}.
The wellposedness of the KdV equation on $\T$ has been extensively studied - c.f. e.g. \cite{Molinet:2012il} for an account on the many results obtained so far. In particular, based on \cite{Bourgain:1993cl} and \cite{Kenig:1996tq} it was proved in \cite{Colliander:2003fv} that the KdV equation is globally uniformly $C^{0}$-wellposed and $C^{\om}$-wellposed on the Sobolev spaces $H_{0}^{s}(\T,\R)$ for any $s\ge -1/2$. In \cite{Kappeler:2006fr} it was shown that the KdV equation is globally $C^{0}$-wellposed in the Sobolev spaces $H_{0}^{s}$, $-1 \le s < 1/2$ and in \cite{Kappeler:2016uj} it was proved that for $-1 < s < -1/2$ and $T > 0$, the solution map $H_{0}^{s}\to C([-T,T],H_{0}^{s})$ is nowhere locally uniformly continuous.
In \cite{Molinet:2012il}, it was shown that the KdV equation is illposed in $H_{0}^{s}$ for $s <-1$. Most closely related to Theorem~\ref{thm:kdv-wp} and Theorem~\ref{thm:invariant} are the wellposedness results of \citet{Bourgain:1997gg} for initial data given by Borel measures which we have already discussed at the beginning of the introduction and the recent wellposedness results in ~\cite{Kappeler:CNzeErmy} on the Fourier Lebesgue spaces $\FL_{0}^{s,p}$ for $-1/2 \le s \le 0$ and $2\le p < \infty$.

\emph{Notation}.
We collect a few notations used throughout the paper. For any $s\in\R$ and $1\le p \le \infty$, denote by $\ell_{0,\C}^{s,p}$ the $\C$-Banach space of complex valued sequences given by
\[
  \ell_{0,\C}^{s,p} \defl
  \setdef{z = (z_{k})_{k\in\Z}\subset\C}{z_{0} = 0;\quad \n{z}_{s,p} < \infty},
\]
where
\[
  \n{z}_{s,p} \defl \p*{\sum_{k\in\Z} \w{n}^{sp}\abs{z_{n}}^{p}}^{1/p},\quad 1\le p < \infty,\qquad
  \n{z}_{s,\infty} \defl \sup_{k\in\Z} \w{n}^{s}\abs{z_{n}},
\]
and by $\ell_{0}^{s,p}$ the real subspace
\[
  \ell_{0}^{s,p} \defl \setdef{z = (z_{k})_{k\in\Z} \in \ell_{0,\C}^{s,p}}{z_{-k} = \ob{z_{k}}\forall k\ge 1}.
\]
By $\ell_{\R}^{s,p}$ we denote the $\R$-subspace of $\ell_{0}^{s,p}$ consisting of real valued sequences $z = (z_{k})_{k\in\Z}$ in $\R$.
Further, we denote by $\FL_{0,\C}^{s,p}$ the Fourier Lebesgue space, introduced by Hörmander,
\[
  \FL_{0,\C}^{s,p} \defl \setdef{q\in S_{\C}'(\T)}{(q_{k})_{k\in\Z}\subset\ell_{0,\C}^{s,p}}
\]
where $q_{k}$, $k\in\Z$, denote the Fourier coefficients of the 1-periodic distribution $q$, $q = \spii{q,e_{k}}$, $e_{k}(x)\defl \e^{\ii k\pi x}$, and $\spii{\cdot,\cdot}$ denotes the $L^{2}$-inner product, $\spii{f,g} = \frac{1}{2}\int_{0}^{2} f(x)\ob{g(x)}\,\dx$, extended by duality to a sesquilinear form on $\Sc_{\C}'(\R/2\Z)\times C_{\C}^{\infty}(\R/2\Z)$.
Correspondingly, we denote by $\FL_{0}^{s,p}$ the real subspace of $\FL_{0,\C}^{s,p}$,
\[
  \FL_{0}^{s,p} \defl \setdef{q\in S_{\C}'(\T)}{(q_{k})_{k\in\Z}\subset\ell_{0}^{s,p}}.
\]
In case $p=2$, we also write $H_{0}^{s}$ [$H_{0,\C}^{s}$] instead of $\FL_{0}^{s,2}$ [$\FL_{0,\C}^{s,2}$] and refer to it as Sobolev space. Similarly, for the sequences spaces $\ell_{0}^{s,2}$ and $\ell_{0,\C}^{s,2}$ we sometimes write $h_{0}^{s}$ [$h_{0,\C}^{s}$].
Occasionally, we will need to consider the sequence spaces $\ell_{\C}^{s,p}(\N) \equiv \ell^{s,p}(\N,\C)$ and $\ell_{\R}^{s,p}(\N) \equiv \ell^{s,p}(\N,\R)$ defined in an obvious way.

Note that for any $z\in \ell_{0}^{s,p}$, $I_{k}\defl z_{k}z_{-k}\ge 0$ for all $k\ge 1$. We denote by $\Tc_{z}$ the torus given by
\[
  \Tc_{z} \defl \setdef{\tilde z = (\tilde z_{k})_{k\in\Z}\in \ell_{0}^{s,p}}{\tilde z_{k}\tilde z_{-k} = z_{k}z_{-k},\quad k\ge 1}.
\]
For $1 \le p < \infty$, $\Tc_{z}$ is compact in $\ell_{0}^{s,p}$ for any $z\in\ell_{0}^{s,p}$ but for $p=\infty$, it is not compact in $\ell^{s,\infty}$ for generic $z$.
For any $s\in\R$ and $1 \le p < \infty$, the dual of $\ell_{0}^{s,p}$ is given by $\ell_{0}^{-s,p'}$ where $p'$ is the conjugate of $p$, given by $1/p+1/p' = 1$. In case $p=1$ we set $p'=\infty$ and in case $p=\infty$ we set $p' = 1$.
We denote by $\tau_{w*}$ the weak* topology on $\ell_{0}^{s,\infty}$ and refer to Appendix~\ref{app:weak-star} for a discussion of the properties of $\tau_{w*}$.

\section{Spectral theory}
\label{s:spectral-theory}

In this section we consider the Schrödinger operator
\begin{equation}
  \label{Lq}
  L(q) = -\partial_{x}^{2} + q,
\end{equation}
which appears in the Lax pair formulation of the \KdV equation. Our aim is to relate the regularity of the potential $q$ to the asymptotic behavior of certain spectral data.

Let $q$ be a \emph{complex potential} in  $H^{-1}_{0,\C} \defl H_{0}^{-1}(\R/\Z,\C)$. In order to treat periodic and antiperiodic boundary conditions at the same time, we consider the differential  operator $L(q) = -\partial_{x}^{2} + q$, on $H^{-1}(\R/2\Z,\C)$ with domain of definition $H^{1}(\R/2\Z,\C)$. See Appendix~\ref{a:hill-op} for a more detailed discussion.
The spectral theory of $L(q)$, while classical for $q \in L_{0,\C}^{2}$, has been only fairly recently extended to the case  $q \in H^{-1}_{0,\C}$ -- see e.g. \cite{Djakov:2009fx,Kappeler:2001bi,Kappeler:2003vh,Korotyaev:2003gp,Savchuk:2003vl,Kappeler:CNzeErmy}
 and the references therein.
The spectrum of $L(q)$,  called the \emph{periodic spectrum of $q$} and denoted by $\spec L(q)$,  is discrete and the eigenvalues, when counted with their multiplicities and ordered lexicographically -- first by their real part and second by their imaginary part -- satisfy
\begin{equation}
  \label{lm-asymptotics-H-1}
  \lm_{0}^{+}(q) \lex \lm_{1}^{-}(q) \lex \lm_{1}^{+}(q) \lex \dotsb,
  \qquad
  \lm_{n}^{\pm}(q) = n^{2}\pi^{2} + n\ell^2_n.
\end{equation}
Furthermore, we define the \emph{gap lengths} $\gm_{n}(q)$ and the \emph{mid points} $\tau_{n}(q)$ by
\begin{equation}
  \label{gm-tau}
  \gm_{n}(q) \defl \lm_{n}^{+}(q)-\lm_{n}^{-}(q),
  \quad
  \tau_{n}(q) \defl \frac{\lm_n^+(q) + \lm_n^-(q)}{2},
  \qquad n\ge 1.
\end{equation}
For $q\in H_{0,\C}^{-1}$ we also consider the operator $L_{\dir}(q)$ defined as the operator $-\partial_{x}^{2} + q$ on $H^{-1}_{\dir}([0,1],\C)$ with domain of definition $H^{1}_{\dir}([0,1],\C)$. See Appendix~\ref{a:hill-op} as well as \cite{Djakov:2009fx,Kappeler:2001bi,Kappeler:2003vh,Korotyaev:2003gp,Savchuk:2003vl} for a more detailed discussion. The spectrum of $L_{\dir}(q)$ is called the \emph{Dirichelt spectrum of $q$}. It is also discrete and given by a sequence of eigenvalues $(\mu_{n})_{n\ge 1}$, counted with multiplicities, which when ordered lexicographically satisfies
\begin{equation}
  \label{mu-asymptotics-H-1}
  \mu_{1} \lex \mu_{2} \lex \mu_{2} \lex \dotsb,\qquad \mu_{n}  = n^{2}\pi^{2} + n\ell^2_n.
\end{equation}

For our purposes we need to characterize the regularity  of potentials $q$ in \emph{weighted Fourier Lebesgue spaces} in terms of the asymptotic behavior of certain spectral quantities.
A normalized, symmetric, monotone, and submultiplicative \emph{weight} is a function $w\colon\Z\to \R$, $n\mapsto w_{n}$, satisfying
\[
  w_{n}\ge 1,\qquad w_{-n} = w_{n},\qquad w_{\abs{n}}\le w_{\abs{n}+1},\qquad w_{n+m}\le w_{n}w_{m},
\]
for all $n,m\in\Z$. The class of all such weights is denoted by $\Ms$.
For $w\in\Ms$, $s\in\R$, and $1\le p \le \infty$, denote by $\FL_{0,\C}^{w,s,p}$ the subspace of $\FL_{0,\C}^{s,p}$ of distributions $f$ whose Fourier coefficients $(f_{n})_{n\in\Z}$ are in the space $\ell_{0,\C}^{w,s,p} = \setdef{ z=(z_n)_{n \in \Z} \in \ell_{0,\C}^{s,p} }{ \n{z}_{w,s,p} < \infty }$ where for $1 \le p < \infty$
\[
  \n{f}_{w,s,p}\defl  \n{(f_n)_{n \in \Z}}_{w,s,p} =  \p*{ \sum_{n\in\Z} w_{n}^{p}\w{n}^{sp} \abs{f_{n}}^{p} }^{1/p},\qquad
  \w{\al}\defl 1+\abs{\al},
\]
and for $p=\infty$,
\[
  \n{f}_{w,s,\infty}\defl  \n{(f_n)_{n \in \Z}}_{w,s,\infty} =  \sup_{n\in\Z}\, w_{n}\w{n}^{s} \abs{f_{n}}.
\]
To simplify notation, we denote the trivial weight $w_{n}\equiv 1$ by $\o$ and write $\FL_{0,\C}^{s,p}\equiv \FL_{0,\C}^{\o,s,p}$.

As a consequence of \eqref{lm-asymptotics-H-1}--\eqref{mu-asymptotics-H-1} it follows that for any $q \in H^{-1}_{0,\C}$,  the sequence of gap lengths $(\gm_{n}(q))_{n\ge 1}$ and the sequence $(\tau_{n}(q)-\mu_{n}(q))_{n\ge 1}$ are both in $\ell^{-1,2}_{\C}(\N)$.
For $q \in \FL^{w,s,\infty}_{0,\C}$, $-1/2<s\le 0$, the sequences have a stronger decay. More precisely, the following results hold:

\begin{thm}
\label{thm:spec-fw}
Let $w\in\Ms$ and $-1/2 < s \le 0$.
\begin{enumerate}[label=(\roman{*})]
\item
\label{fw:per}
For any $q\in\FL_{0,\C}^{w,s,\infty}$, one has $(\gm_{n}(q))_{n\ge 1}\in\ell_{\C}^{w,s,\infty}(\N)$ and the map
\[
  \FL_{0,\C}^{w,s,\infty}\to \ell_{\C}^{w,s,\infty}(\N),\qquad q \mapsto (\gm_{n}(q))_{n\ge 1},
\]
is locally bounded.

\item
\label{fw:dir}
For any $q\in\FL_{0,\C}^{w,s,\infty}$, one has $(\tau_{n}-\mu_{n}(q))_{n\ge 1}\in\ell_{\C}^{w,s,\infty}(\N)$ and the map
\[
  \FL_{0,\C}^{w,s,\infty}\to \ell_{\C}^{w,s,\infty}(\N),\qquad q \mapsto (\tau_{n}(q)-\mu_{n}(q))_{n\ge 1},
\]
is locally bounded.~\fish
\end{enumerate}
\end{thm}

A key ingredient for studying the restriction of the Birkhoff map of the KdV equation, defined on $H_{0}^{-1}$, to $\FL_{0}^{s,\infty}$ is the following spectral characterization for a potential $q\in H_{0}^{-1}$ to be in $\FL_{0}^{s,\infty}$.

\begin{thm}
\label{inv}
Let $q\in H_{0}^{-1}$ with gap lengths $\gm(q)\in \ell_{\R}^{s,\infty}$ for some $-1/2< s \le 0$. Then the following holds:
\begin{enumerate}[label=(\roman{*})]
\item
$q \in \FL_{0}^{s,\infty}$.
\item
$\Iso(q) \subset \FL_{0}^{s,\infty}$.
\item
$\Iso(q)$ is weak* compact.~\fish
\end{enumerate}
\end{thm}

\begin{rem}
For any $-1/2 < s \le 0$, there are potentials $q\in \FL^{s,\infty}$ so that $\Iso(q)$ is not compact in $\FL^{s,\infty}$ -- see item (iii) in Lemma~\ref{iso-fl-s-infty}.~\map
\end{rem}

In the remainder of this section we prove Theorem~\ref{thm:spec-fw} and Theorem~\ref{inv} by extending the methods, used in \cite{Kappeler:2001hsa,Poschel:2011iua,Djakov:2006ba} for potentials $q \in L^2$, for singular potentials. We point out that the spectral theory is only developed as far as needed.

\subsection{Setup}

We extend the $L^{2}$-inner product $\spii{f,g} = \frac{1}{2}\int_{0}^{2} f(x)\ob{g(x)}\,\dx$ on $L_{\C}^{2}(\R/2\Z)\equiv L^{2}(\R/2\Z,\C)$ by duality to $\Sc_{\C}'(\R/2\Z)\times C_{\C}^{\infty}(\R/2\Z)$. Let $e_{n}(x) = \e^{\ii \pi n x}$, $n\in\Z$, and for $w\in\Ms$, $s\in\R$, and $1\le p \le \infty$ denote by $\FL_{\star,\C}^{w,s,p}$ the space of $2$-periodic, complex valued distributions $f\in \Sc_{\C}'(\R/2\Z)$ so that the sequence of their  Fourier coefficients $f_{n} = \spii{f,e_{n}}$ is in the space $\ell^{w,s,p}_\C = \setdef{ z=(z_n)_{n \in \Z} \subset \C }{ \n{z}_{w,s,p} < \infty }$.
To simplify notation, we write $\FL_{\star,\C}^{s,p}\equiv \FL_{\star,\C}^{\o,s,p}$.

In the sequel we will identify a potential $q\in \FL_{0,\C}^{w,s,\infty}$ with the corresponding element $\sum_{n \in \Z} q_{n}e_{n}$ in $\FL_{\star,\C}^{w,s,\infty}$ where $q_{n}$ is the $n$th Fourier coefficient of the potential obtained from $q$ by viewing it as a distribution on $\R/2\Z$ instead of $\R/\Z$, i.e., $q_{2n} = \spii{q, e_{2n}}$, whereas $q_{2n+1} = \spii{q, e_{2n+1}} = 0$ and $q_0 = \spii{q,1} = 0$. We denote by $V$ the operator of multiplication by $q$ with domain $H_{\C}^{1}(\R/2\Z)$. See Appendix~\ref{a:hill-op} for a detailed discussion of this operator as well as the operator $L(q)$ introduced in~\eqref{Lq}. When expressed in its Fourier series, the image $Vf$ of $f = \sum_{n\in\Z} f_{n}e_{n}\in H_{\C}^{1}(\R/2\Z)$ is the distribution $Vf = \sum_{n\in \Z}\p*{ \sum_{m\in \Z} q_{n-m}f_{m} }e_{n}\in H_{\C}^{-1}(\R/2\Z)$. To prove the asymptotic estimates of the gap lengths stated in Theorem~\ref{thm:spec-fw} we need to study the eigenvalue equation $L(q)f = \lm f$ for sufficiently large periodic eigenvalues $\lm$. For $q\in H_{0,\C}^{-1}$, the domain of $L(q)$ is $H_{\C}^{1}(\R/2\Z)$ and hence the eigenfunction $f$ is an element of this space. It is shown in Appendix~\ref{a:hill-op} that for $q\in \FL^{s,\infty}_{0,\C}$ with $-1/2<s\le 0$ and $2\le p \le \infty$, one has $f\in \FL_{\star,\C}^{s+2,p}$ and $\partial_{x}^{2}f$, $Vf\in \FL_{\star,\C}^{s,\infty}$. Note that for $q = 0$ and any $n\ge 1$, $\lm_{n}^{+}(0) = \lm_{n}^{-}(0) = n^{2}\pi^{2}$, and the eigenspace corresponding to the double eigenvalue $\lm_{n}^{+}(0) = \lm_{n}^{-}(0)$ is spanned by $e_{n}$ and $e_{-n}$. Viewing $L(q)-\lm_{n}^{\pm}(q)$ for $n$ large as a perturbation of $L(0)-\lm_{n}^{\pm}(0)$, we are led to decompose $\FL_{\star,\C}^{s,\infty}$ into the direct sum $\FL_{\star,\C}^{s,\infty} = \Pc_{n}\oplus\Qc_{n}$,
\begin{equation}
  \label{Pn-Qn}
  \Pc_{n} = \spanx\setd{e_{n},\, e_{-n}},\qquad
  \Qc_{n} = \overline{\spanx}\setdef{e_{k}}{k\neq \pm n}.
\end{equation}
The $L^{2}$-orthogonal projections onto $\Pc_{n}$ and $\Qc_{n}$ are denoted by $P_{n}$ and $Q_{n}$, respectively. It is convenient to write the eigenvalue equation $Lf = \lm f$ in the form $A_{\lm}f = Vf$, where $A_{\lm}f = \partial_{x}^{2}f + \lm f$ and $V$ denotes the operator of multiplication with $q$. Since $A_{\lm}$ is a Fourier multiplier, we write $f = u+v$, where $u = P_{n}f$ and $v = Q_{n}f$, and decompose the equation $A_{\lm}f = V f$ into the two equations
\begin{equation}
  \label{P-Q-eqn}
  A_{\lm}u = P_{n}V(u+v),\qquad
  A_{\lm}v = Q_{n}V(u+v),
\end{equation}
referred to as $P$- and $Q$-equation. 
Since $q\in H_{0,\C}^{-1}$, it follows from~\cite{Kappeler:2003vh} that $\lm_{n}^{\pm}(q) = n^{2}\pi^{2} + n\ell_{n}^{2}$. Hence for $n$ sufficiently large, $\lm_{n}^{\pm}(q)\in S_{n}$ where $S_{n}$ denotes the closed vertical strip
\begin{equation}
  \label{Sn}
  S_{n} \defl \setdef{\lm\in \C}{\abs{\Re \lm - n^{2}\pi^{2}} \le 12 n},\qquad n\ge 1.
\end{equation}
Note that $\setdef{\lm\in \C}{\Re \lm \ge 0}\subset \bigcup_{n\ge 1} S_{n}$. Given any $n\ge 1$, $u\in \Pc_{n}$, and $\lm\in S_{n}$, we derive in a first step from the $Q$-equation an equation for $Vv$ which for $n$ sufficiently large can be solved as a function of $u$ and $\lm$.
In a second step, for $\lm$ a periodic eigenvalue in $S_{n}$, we solve the $P$ equation for $u$ after having substituted in it the expression of $Vv$.
The solution of the $Q$-equation is then easily determined.
Towards the first step note that for any $\lm\in S_{n}$, $A_{\lm}\colon \Qc_{n}\cap \FL_{\star,\C}^{s+2,p}\to \Qc_{n}$ is boundedly invertible as for any $k\neq n$,
\begin{equation}
  \label{lm-n2k2-lower-bound}
  \min_{\lm\in S_{n}}\abs{\lm-k^{2}\pi^{2}}
   \ge \min_{\lm\in S_{n}} \abs{\Re \lm - k^{2}\pi^{2}}
   \ge \abs{n^{2}-k^{2}} \ge 1.
\end{equation}
In order to derive from the $Q$-equation an equation for $Vv$, we apply to it the operator $VA_{\lm}^{-1}$ to get
\[
  Vv = VA_{\lm}^{-1}Q_{n}V(u+v) = T_{n}V(u+v),
\]
where
\[
  T_{n}\equiv T_{n}(\lm) \defl VA_{\lm}^{-1}Q_{n}
  \colon \FL_{\star,\C}^{w,s,\infty}\to \FL_{\star,\C}^{w,s,\infty}.
\]
It leads to the following equation for $\check{v} \defl Vv$
\begin{equation}
  \label{w-eqn}
  (\Id-T_{n}(\lm))\check{v} = T_{n}(\lm)Vu.
\end{equation}
To show that $\Id-T_{n}(\lm)$ is invertible, we introduce for any $s\in\R$, $w\in\Ms$, and $l\in\Z$ the shifted norm of $f\in\FL_{\star,\C}^{w,s,\infty}$,
\begin{equation*}
  \n{f}_{w,s,\infty;l} \defl \n{fe_{l}}_{w,s,\infty}
   = \n*{ \p[\big]{ w_{k+l}\lin{k+l}^{s} f_{k}}_{k\in\Z} }_{\ell^{p}},
\end{equation*}
and denote by $\n{T_{n}}_{w,s,\infty;l}$ the operator norm of $T_{n}$ viewed as an operator on $\FL_{\star,\C}^{w,s,\infty}$ with norm $\n{\cdot}_{w,s,\infty;l}$.
Furthermore, we denote by $R_{N}f$, $N\ge 1$, the tail of the Fourier series of $f\in \FL_{\star,\C}^{w,s,\infty}$,
\[
  R_{N}f = \sum_{\abs{k}\ge N} f_{k}e_{k}.
\]

\begin{lem}
\label{Tn-est}
Let $-1/2 < s \le 0$, $w\in\Ms$, and $n\ge 1$ be given. For any $q\in\FL_{0,\C}^{w,s,\infty}$ and $\lm\in S_{n}$,
\[
  T_{n}(\lm)\colon \FL_{\star,\C}^{w,s,\infty}\to \FL_{\star,\C}^{w,s,\infty}
\]
is a bounded linear operator satisfying the estimate
\begin{equation}
  \label{Tn-est-exp}
  \n{T_{n}(\lm)}_{w,s,\infty;\pm n} \le \frac{c_{s}}{n^{1/2-\abs{s}}}\n{q}_{w,s,\infty},
\end{equation}
where $c_{s}\ge1$ is a constant depending only on and decreasing monotonically in $s$.
In particular, $c_{s}$ does not depend on $q$ nor on the weight $w$.
\fish
\end{lem}

\begin{proof}
Let $s$ and $w$ be given as in the statement of the lemma.
Note that $A_{\lm}^{-1}\colon \FL_{\star,\C}^{s,\infty}\to \FL_{\star,\C}^{s+2,\infty}$ is bounded for any $\lm\in S_{n}$ and hence for any $f\in \FL_{\star,\C}^{s,\infty}$, $q\in \FL_{0,\C}^{s,\infty}$, and $\lm\in S_{n}$, the multiplication of $A_{\lm}^{-1}Q_{n}f$ with $q$, defined by
\begin{equation}
  \label{VAQ}
  VA_{\lm}^{-1}Q_{n}f = \sum_{m\in\Z}\p*{ \sum_{\abs{k}\neq n} \frac{q_{m-k} f_{k}}{\lm-k^{2}\pi^{2}} }e_{m}
\end{equation}
is a distribution in $S_{\C}'(\R/2\Z)$. Note that $T_{n}(\lm)f = VA_{\lm}^{-1}Q_{n}f$ and that its  norm $\n{T_{n}(\lm)f}_{w,s,\infty;n}$ satisfies for any $\lm\in S_{n}$,
\begin{align*}
  \n{T_{n}f}_{w,s,\infty;n} 
  &\le
  \sup_{m\in\Z} \sum_{\abs{k}\neq n} 
    \frac{w_{m+n}\w{k+n}^{\abs{s}}  \w{m-k}^{\abs{s}}}
         {\abs{n+k}\abs{n-k} \w{m+n}^{\abs{s}}}
    \frac{\abs{q_{m-k}}} {\w{m-k}^{\abs{s}}}
    \frac{\abs{f_{k}}}   {\w{k+n}^{\abs{s}}}
   ,
\end{align*}
where we have used~\eqref{lm-n2k2-lower-bound}.
Since $\w{m-k}\le \w{m+n}\w{n+k}$, $-1/2 < s\le 0$, and $\w{\nu}/\abs{\nu} \le 2$, we conclude
\[
  \frac{\w{k+n}^{\abs{s}}  \w{m-k}^{\abs{s}}}
         {\abs{n+k}\abs{n-k} \w{m+n}^{\abs{s}}}
  \le
  \frac{\w{k+n}^{2\abs{s}}}
         {\abs{n+k}\abs{n-k} }
  \le
  \frac{2}
         {\abs{n+k}^{1-2\abs{s}}\abs{n-k}}.
\]
Hölder's inequality together with the submultiplicativity of the weight $w$ then yields
\begin{align*}
  \n{T_{n}f}_{w,s,\infty;n}
 &\le 2\sup_{m\in\Z}
             \sum_{\abs{k}\neq n} 
             \frac{1}
         {\abs{n+k}^{1-2\abs{s}}\abs{n-k}}
             \frac{w_{m-k}\abs{q_{m-k}}}{\w{m-k}^{\abs{s}}}
             \frac{w_{k+n}\abs{f_{k}}}{\w{k+n}^{\abs{s}}}         
   \\
 &\le 2\p*{\sum_{\abs{k}\neq n} 
             \frac{1}
         {\abs{n+k}^{(1-2\abs{s})}\abs{n-k}}}\n{q}_{w,s,\infty}\n{f}_{w,s,\infty;n}.
\end{align*}
One checks that
\begin{align*}
  \sum_{\abs{k}\neq n} 
             \frac{1}
         {\abs{n+k}^{(1-2\abs{s})}\abs{n-k}}\le \frac{c_{s}}{n^{1/2-\abs{s}}},
\end{align*}
Going through the arguments of the proof one sees that the same kind of estimates also lead to the claimed bound for $\n{T_{n}f}_{w,s,\infty;-n}$.~\qed
\end{proof}

Lemma~\ref{Tn-est} can be used to solve, for $n$ sufficiently large, the equation \eqref{w-eqn} as well as the $Q$-equation \eqref{P-Q-eqn} in terms of any given $u\in\Pc_{n}$ and $\lm\in S_{n}$.

\begin{cor}
\label{Q-soln}
For any $q\in\FL_{0,\C}^{w,s,\infty}$ with $-1/2 < s \le 0$ and $w\in\Ms$, there exists $n_{s} = n_{s}(q) \ge 1$ so that,
\begin{equation}
  \label{cs-est}
  2c_{s}\n{q}_{w,s,\infty} \le n_{s}^{1/2-\abs{s}},
\end{equation}
with $c_{s}\ge 1$ the constant in~\eqref{Tn-est-exp} implying that for any $\lm\in S_{n}$, $T_{n}(\lm)$ is a $1/2$ contraction on $\FL_{\star,\C}^{w,s,\infty}$ with respect to the norms shifted by $\pm n$, $\n{T_{n}(\lm)}_{w,s,\infty;\pm n}\le 1/2$. The threshold $n_{s}(q)$ can be chosen uniformly in $q$ on bounded subsets of~$\FL_{0,\C}^{w,s,\infty}$. As a consequence, for $n\ge n_{s}(q)$, equation~\eqref{w-eqn} and \eqref{P-Q-eqn} can be uniquely solved for any given $u\in\Pc_{n}$, $\lm\in S_{n}$,
\begin{align}
  \label{sol-w-eqn}
  \check{v}_{u,\lm} &= K_{n}(\lm)T_{n}(\lm)Vu \in \FL_{\star,\C}^{s,\infty},
  \qquad K_{n} \equiv K_{n}(\lm) \defl (\Id - T_{n}(\lm))^{-1},\\
  \label{sol-Q-eqn}
    v_{u,\lm} &= A_{\lm}^{-1}Q_{n}Vu + A_{\lm}^{-1}Q_{n}\check{v}_{u,\lm}
    = 
  A_{\lm}^{-1}Q_{n}K_{n}Vu \in \FL_{\star,\C}^{s+2,\infty}\cap \Qc_{n}.
\end{align}
In particular, one has $\check{v}_{u,\lm} = Vv_{u,\lm}$.~\fish
\end{cor}

\begin{rem}
\label{inhom-1}
By the same approach, one can study the inhomogeneous equation
\[
  (L-\lm)f = g,\qquad g\in \FL_{\star,\C}^{s,\infty},
\]
for $\lm\in S_{n}$ and $n\ge n_{s}$. Writing $f=u+v$ and $g = P_{n}g + Q_{n}g$, the $Q$-equation becomes
\[
  A_{\lm}v = Q_{n}V(u+v)- Q_{n}g = Q_{n}Vv + Q_{n}(Vu-g)
\]
leading for any given $u\in\Pc_{n}$ and $\lm\in S_{n}$ to the unique solution $\check{v}$ of the equation corresponding to~\eqref{w-eqn}
\[
  \check{v} = Vv = K_{n}T_{n}(Vu-g) \in \FL_{\star,\C}^{s,\infty},
\]
and, in turn, to the unique solution $v \in \FL_{\star,\C}^{s+2,\infty}\cap \Qc_{n}$ of the $Q$-equation
\[
  v = A_{\lm}^{-1}Q_{n}(Vu - g) + A_{\lm}^{-1}Q_{n}K_{n}T_{n}(Vu-g)  = A_{\lm}^{-1}Q_{n}K_{n}(Vu-g).~\map
\]
\end{rem}

\subsection{Reduction}
\label{ss:spec-reduction}

In a next step we study the $P$-equation $A_{\lm}u = P_{n}V(u+v)$ of \eqref{P-Q-eqn}. For $n\ge n_{s}(q)$, $u\in \Pc_{n}$, and $\lm\in S_{n}$, substitute in it the solution $\check{v}_{u,\lm}$ of \eqref{w-eqn}, given by~\eqref{sol-w-eqn},
\[
  A_{\lm}u = P_{n}Vu + P_{n}\check{v}_{u,\lm} = P_{n}(\Id + K_{n}T_{n})Vu.
\]
Using that $\Id + K_{n}T_{n} = K_{n}$ one then obtains $A_{\lm}u = P_{n}K_{n}Vu$ or $B_{n}u = 0$, where
\begin{equation}
  B_{n} \equiv B_{n}(\lm)\colon \Pc_{n}\to \Pc_{n},\quad u\mapsto (A_{\lm}-P_{n}K_{n}(\lm)V)u.
\end{equation}

\begin{lem}
\label{eval-char}
Assume that $q\in \FL_{0,\C}^{s,\infty}$ with $-1/2 < s \le 0$. Then for any $n\ge n_{s}$ with $n_{s}$ given by Corollary~\ref{Q-soln}, $\lm\in S_{n}$ is an eigenvalue of $L(q)$ if and only if $\det(B_{n}(\lm)) = 0$.~\fish
\end{lem}

\begin{proof}
Assume that $\lm\in S_{n}$ is an eigenvalue of $L=L(q)$. By Lemma~\ref{regularity-inhomogeneous} there exists $0\neq f\in \FL_{\star,\C}^{s+2,\infty}$ so that $Lf = \lm f$. Decomposing $f = u+v\in \Pc_{n}\oplus \Qc_{n}$ it follows by the considerations above and the assumption $n\ge n_{s}$ that $u\neq 0$ and $B_{n}(\lm)u = 0$. Conversely, assume that $\det(B_{n}(\lm)) = 0$ for some $\lm\in S_{n}$. Then there exists $0\neq u\in \Pc_{n}$ so that $B_{n}(\lm)u = 0$. Since $n\ge n_{s}$, there exist $\check{v}_{u,\lm}$ and $v_{u,\lm}$ as in \eqref{sol-w-eqn} and \eqref{sol-Q-eqn}, respectively. Then $v \equiv v_{u,\lm}\in \FL_{\star,\C}^{s+2,\infty}\cap \Qc_{n}$ solves the $Q$-equation by Corollary~\ref{Q-soln}. To see that $u$ solves the $P$-equation, note that $B_{n}(\lm)u = 0$ implies that
\[
  A_{\lm}u = P_{n}K_{n}Vu = P_{n}(\Id + K_{n}T_{n})Vu.
\]
As by \eqref{sol-w-eqn}, $\check{v}_{u,\lm} = K_{n}T_{n}Vu$, and as $\check{v}_{u,\lm} = Vv$, one sees that indeed
\[
  A_{\lm}u = P_{n}Vu + P_{n}Vv.~\qed
\]
\end{proof}

\begin{rem}
\label{inhom-2}
Solutions of the inhomogeneous equation $(L-\lm)f = g$ for $g\in\FL_{\star,\C}^{s,\infty}$, $\lm\in S_{n}$, and $n\ge n_{s}$ can be obtained by substituting into the $P$-equation
\[
  A_{\lm}u = P_{n}Vu + P_{n}Vv - P_{n}g
\]
the expression for $Vv$ obtained in Remark~\ref{inhom-1}, $Vv = K_{n}T_{n}(Vu-g)$, to get
\begin{align*}
  A_{\lm}u &= P_{n}Vu + P_{n}K_{n}T_{n}Vu - P_{n}g - P_{n}K_{n}T_{n}g\\
  &= P_{n}(\Id + K_{n}T_{n})Vu - P_{n}(\Id + K_{n}T_{n})g.
\end{align*}
Using that $\Id + K_{n}T_{n} = K_{n}$ one concludes that
\begin{equation}
  \label{B-inhom}
  B_{n}(\lm)u = -P_{n}K_{n}(\lm)g.
\end{equation}
Conversely, for any solution $u$ of \eqref{B-inhom}, $f = u+v$, with $v$ being the element in $\FL_{\star,\C}^{s+2,\infty}$ given in Remark~\ref{inhom-1}, satisfies $(L-\lm)f = g$ and $f\in \FL_{\star,\C}^{s+2,\infty}$.~\map
\end{rem}

We denote the matrix representation of a linear operator $F\colon \Pc_{n}\to \Pc_{n}$ with respect to the orthonormal basis $e_{n}$, $e_{-n}$ of $\Pc_{n}$ also by $F$,
\[
  F = \begin{pmatrix}
  \spii{Fe_{n},e_{n}}  & \spii{Fe_{-n},e_{n}}\\
  \spii{Fe_{n},e_{-n}}  & \spii{Fe_{-n},e_{-n}}
  \end{pmatrix}.
\]
In particular,
\[
  A_{\lm} = 
  \begin{pmatrix}
  \lm-n^{2}\pi^{2} & 0\\
  0 & \lm-n^{2}\pi^{2}
  \end{pmatrix},
  \qquad
  P_{n}K_{n}V = 
  \begin{pmatrix}
  a_{n} & b_{n}\\
  b_{-n} & a_{-n}
  \end{pmatrix},
\]
where for any $\lm\in S_{n}$ and $n\ge n_{s}$ the coefficients of $P_{n}K_{n}V$ are given  by
\begin{align*}
  a_{n}  &\equiv a_{n}(\lm)  \defl \spii{K_{n}Ve_{n},e_{n}},&
  a_{-n} &\equiv a_{-n}(\lm) \defl \spii{K_{n}Ve_{-n},e_{-n}},\\
  b_{n}  &\equiv b_{n}(\lm)  \defl \spii{K_{n}Ve_{-n},e_{n}},&
  b_{-n} &\equiv b_{-n}(\lm) \defl \spii{K_{n}Ve_{n},e_{-n}}.
\end{align*}
Note that for any $\lm\in S_{n}$, the functions $a_{\pm n}(\lm)$ and $b_{\pm n}(\lm)$ have the following series expansion
\begin{equation}
  \label{an-bn-expansions}
  a_{\pm n}(\lm) = \sum_{l\ge 0} \lin{T_{n}(\lm)^{l}Ve_{\pm n},e_{\pm n}},\qquad
  b_{\pm n}(\lm) = \sum_{l\ge 0} \lin{T_{n}(\lm)^{l}Ve_{\mp n},e_{\pm n}}.
\end{equation}
Furthermore, by a straightforward verification it follows from the expression of $a_{n}$ in terms of the representation of $K_{n} = \sum_{k\ge 0} T_{n}(\lm)^{k}$ and $V$ in Fourier space that for any $n\ge n_{s}$
\begin{equation}
  \label{an-symmetry}
  a_{n} = \spii{K_{n}Ve_{-n},e_{-n}} = a_{-n}.
\end{equation}
Hence,
\begin{equation}
  \label{Bn-mat}
   B_{n}(\lm) = \begin{pmatrix}
            \lm-n^{2}\pi^{2} - a_{n}(\lm) & -b_{n}(\lm),\\
           -b_{-n}(\lm)                   &  \lm-n^{2}\pi^{2} - a_{n}(\lm)
           \end{pmatrix}.
\end{equation}
In addition, if $q$ is real valued, then
\begin{equation}
  \label{bn-symmetry}
  a_{n}(\ob{\lm}) = a_{n}(\lm),\qquad  b_{-n}(\ob{\lm}) = \ob{b_{n}(\lm)},\qquad \lm\in S_{n}.
\end{equation}

\begin{lem}
\label{coeff-est}
Suppose $q\in \FL_{0,\C}^{w,s,\infty}$ with $-1/2< s\le 0$ and $w\in\Ms$. Then for any $n\ge n_{s}$, with $n_{s}$ as in Corollary~\ref{Q-soln}, the coefficients $a_{n}(\lm)$ and $b_{\pm n}(\lm)$ are analytic functions on the strip $S_{n}$ and for any $\lm\in S_{n}$
\begin{align*}
  \text{(i)}\quad&\abs{a_{n}(\lm)}            \le 2\n{T_{n}(\lm)}_{w,s,\infty;\pm n}       \n{q}_{s,\infty},\\
  \text{(ii)}\quad&w_{2n}\w{2n}^{s}\abs{b_{\pm n}(\lm)  - q_{\pm 2n}}  \le 2\n{T_{n}(\lm)}_{w,s,\infty;\pm n}  \n{q}_{w,s,\infty}.~\fish
\end{align*}
\end{lem}

\begin{proof}
Let us first prove the claimed estimate for $\abs{b_{n}(\lm)-q_{2n}}$.
Since $\n{T_{n}(\lm)}_{w,s,\infty;n}\le 1/2$ for $n\ge n_{s}$ and $\lm\in S_{n}$, the series expansion~\eqref{an-bn-expansions} of $b_{n}$ converges uniformly on $S_{n}$ to an analytic function in $\lm$. Moreover, we obtain from the identity $K_{n} = \Id + T_{n}K_{n}$
\[
  b_{n} = \spii{Ve_{-n},e_{n}} + \spii{T_{n}K_{n}Ve_{-n},e_{n}}
  = q_{2n} + \spii{T_{n}K_{n}Ve_{-n},e_{n}}.
\]
Furthermore, for any $f\in \FL_{\star,\C}^{w,s,\infty}$ we compute
\[
  w_{2n}\w{2n}^{s}\abs{\spii{f,e_{n}}}
   = w_{2n}\w{2n}^{s}\abs{\spii{fe_{n},e_{2n}}}
  \le \n{fe_{n}}_{w,s,p} = \n{f}_{w,s,p;n}.
\]
Consequently, using that $\n{T_{n}}_{w,s,p;n} \le 1/2$ and hence $\n{K_{n}}_{w,s,p;n} \le 2$, one gets
\begin{align*}
  w_{2n}\w{2n}^{s}\abs{b_{n}-q_{2n}} 
   \le \n{T_{n}K_{n}Ve_{-n}}_{w,s,p;n}
  &\le 2\n{T_{n}}_{w,s,p;n}\n{Ve_{-n}}_{w,s,p;n}\\
  &= 2\n{T_{n}}_{w,s,p;n}\n{q}_{w,s,p}.
\end{align*}
The estimates for $\abs{b_{-n}-q_{-2n}}$ and $\abs{a_{n}}$ are obtained in a similar fashion.~\qed
\end{proof}

The following refined estimate will be needed in the proof of Lemma~\ref{En-basis} in Subsection~\ref{ss:jordan}.

\begin{lem}
\label{bn-est}
Let $q\in \FL_{0,\C}^{w,s,\infty}$ with $w\in\Ms$ and $-1/2 < s\le 0$.
Then for any $f\in\FL_{\star,\C}^{s,\infty}$ and $\lm\in S_{n}$ with $n\ge n_{s}$,
\[
  w_{2n}\w{2n}^{s}\abs{\spii{T_{n}f,e_{\pm n}}}
   \le c_{s}'\ep_{s}(n)\n{q}_{w,s,p} \n{f}_{w,s,p;\pm n}
\]
where $c_{s}'\ge c_{s}\ge 1$ is independent of $q$, $n$, and $\lm$, and
\[
  \ep_{s}(n) = 
  \begin{cases}
  \frac{\log\w{n}}{n}, & s = 0,\\
  \frac{1}{n^{1-\abs{s}}}, & -1/2 < s < 0.~\fish
  \end{cases}
\]
\end{lem}

\begin{proof}
As the estimates of $\spii{T_{n}f,e_{n}}$ and $\spii{T_{n}f,e_{-n}}$ can be proved in a similar way we concentrate on $\spii{T_{n}f,e_{n}}$. Since by definition $T_{n} = VA_{\lm}^{-1}Q_{n}$,
\[
  \spii{T_{n}f,e_{n}} = \sum_{\abs{m}\neq n} \frac{q_{n-m}f_{m}}{\lm-m^{2}\pi^{2}}.
\]
Using that $\w{n+m}/\abs{n+m}$, $\w{n-m}/\abs{n-m}\le 2$ for $\abs{m}\neq n$ together with~\eqref{lm-n2k2-lower-bound}, and the submultiplicativity of the weight,  one gets for any $\lm\in S_{n}$,
\begin{align*}
  w_{2n}\lin{2n}^{s}\abs{\spii{T_{n}f,e_{n}}} 
  &\le 2\sum_{\abs{m}\neq n} 
  \frac{\lin{2n}^{s}}{\abs{n^{2}-m^{2}}^{1-\abs{s}}}
  \frac{w_{n-m}\abs{q_{n-m}}}{\w{n-m}^{\abs{s}}}
  \frac{w_{n+m}\abs{f_{m}}}{\w{n+m}^{\abs{s}}}\\
  &\le
  2\p*{\sum_{\abs{m}\neq n} 
  \frac{\lin{2n}^{s}}{\abs{n^{2}-m^{2}}^{1-\abs{s}}}}
  \n{q}_{w,s,\infty}\n{f}_{w,s,\infty;n}.
\end{align*}
Finally, by Lemma~\ref{hilbert-sum},
\[
  \sum_{\abs{m}\neq n} 
  \frac{1}{\abs{n^{2}-m^{2}}^{1-\abs{s}}}
  \le
  \begin{cases}
  \frac{\tilde c_{s}\log \w{n}}{n}, & s = 0,\\
  \frac{\tilde c_{s}}{n^{1-2\abs{s}}}, & -1/2 < s < 0.
  \end{cases}
\]
Altogether we thus have proved the claim.~\qed
\end{proof}

The preceding lemma together with~\eqref{cs-est} implies that for any $n\ge n_{s}$, the function $\det B_{n}(\lm) = (\lm-n^{2}\pi^{2}-a_{n})^{2} - b_{n}b_{-n}$ is analytic in $\lm\in S_{n}$ and can be considered a small perturbation of $(\lm-n^{2}\pi^{2})^{2}$ provided $n\ge n_{s}$ is sufficiently large.

\begin{lem}
\label{Sn-roots}
Suppose $q\in \FL_{0,\C}^{s,\infty}$ with $-1/2< s \le 0$. Choose $n_{s} = n_{s}(q) \ge 1$ as in Corollary~\ref{Q-soln}.
Then for any $n\ge n_{s}$, $\det(B_{n}(\lm))$ has exactly two roots $\xi_{n,1}$ and $\xi_{n,2}$ in $S_{n}$ counted with multiplicity. They are contained in
\[
  D_{n} \defl \setdef{\lm}{\abs{\lm-n^{2}\pi^{2}} \le 4n^{1/2}} \subset S_{n}
\]
and satisfy
\begin{equation}
  \label{xi-est}
  \abs{\xi_{n,1}-\xi_{n,2}} \le \sqrt{6}\sup_{\lm\in S_{n}} \abs{b_{n}(\lm)b_{-n}(\lm)}^{1/2}.~\fish
\end{equation}
\end{lem}

\begin{proof}
Since for any $n\ge n_{s}$ and $\lm\in S_{n}$, $\n{T_{n}(\lm)}_{s,\infty;\pm n} \le 1/2$, one concludes from the preceding lemma that $\abs{a_{n}(\lm)} \le \n{q}_{s,\infty}$ and, with $\abs{b_{\pm n}(\lm)} \le \abs{q_{\pm 2n}} + \abs{b_{\pm n}(\lm)-q_{\pm 2n}}$, that
\[
  \lin{2n}^{s}\abs{b_{\pm n}(\lm)} \le 2\n{q}_{s,\infty}.
\]
Furthermore, by~\eqref{cs-est},
\[
  2\n{q}_{s,\infty} \le n_{s}^{1/2-\abs{s}}.
\]
Therefore, for any $\lm,\mu\in S_{n}$,
\begin{equation}
  \label{an-cn-est}
  \begin{split}
    \abs{a_{n}(\mu)} + \abs{b_{n}(\lm)b_{-n}(\lm)}^{1/2} &\le 
  \p*{ 1 + 2\w{2n}^{\abs{s}} }\n{q}_{s,\infty}\\
  &< 6n^{\abs{s}}\n{q}_{s,\infty} \le 4n^{1/2}
   = \inf_{\lm\in \partial D_{n}} \abs{\lm - n^{2}\pi^{2}}.
  \end{split}
\end{equation}
It then follows that $\det B_{n}(\lm)$ has no root in $S_{n}\setminus D_{n}$. Indeed, assume that $\xi\in S_{n}$ is a root, then $\abs{\xi - n^{2}\pi^{2}-a_{n}(\xi)} = \abs{b_{n}(\xi)b_{-n}(\xi)}^{1/2}$ and hence
\[
  \abs{\xi-n^{2}\pi^{2}}
   \le \abs{a_{n}(\xi)} + \abs{b_{n}(\xi)b_{-n}(\xi)}^{1/2}
    < 4n^{1/2},
\]
implying that $\xi\in D_{n}$. In addition, \eqref{an-cn-est} implies that by Rouché's theorem the two analytic functions $\lm-n^{2}\pi^{2}$ and $\lm-n^{2}\pi^{2} -a_{n}(\lm)$, defined on the strip $S_{n}$ have the same number of roots in $D_{n}$ when counted with multiplicities. As a consequence $(\lm-n^{2}\pi^{2}-a_{n}(\lm))^{2}$ has a double root in $D_{n}$. Finally, \eqref{an-cn-est} also implies that
\begin{align*}
  \sup_{\lm\in S_{n}} \abs{b_{n}(\lm)b_{-n}(\lm)}^{1/2} 
  &< \inf_{\lm\in \partial D_{n}} \abs{\lm-n^{2}\pi^{2 }} - \sup_{\lm\in S_{n}} \abs{a_{n}(\lm)}\\
  &\le \inf_{\lm\in \partial D_{n}} \abs{\lm-n^{2}\pi^{2} - a_{n}(\lm)}
\end{align*}
and hence again by Rouché's theorem, the analytic functions $(\lm-n^{2}\pi^{2}-a_{n}(\lm))^{2}$ and $(\lm-n^{2}\pi^{2}-a_{n}(\lm))^{2}-b_{n}(\lm)b_{-n}(\lm)$ have the same number of roots in $D_{n}$.
Altogether we thus have established that $\det(B_{n}(\lm)) = (\lm-n^{2}\pi^{2}-a_{n}(\lm))^{2} - b_{n}(\lm)b_{-n}(\lm)$ has precisely two roots $\xi_{n,1}$, $\xi_{n,2}$ in $D_{n}$.

To estimate the distance of the roots, write $\det B_{n}(\lm)$ as a product $g_{+}(\lm)g_{-}(\lm)$ where $g_{\pm}(\lm) = \lm-n^{2}\pi^{2} - a_{n}(\lm) \mp \ph_{n}(\lm)$
and $\ph_{n}(\lm) = \sqrt{b_{n}(\lm)b_{-n}(\lm)}$ with an arbitrary choice of the sign of the root for any $\lm$. Each root $\xi$ of $\det(B_{n})$ is either a root of $g_{+}$ or $g_{-}$ and thus satisfies
\[
  \xi \in \setd{n^{2}\pi^{2} + a_{n}(\xi) \pm \ph_{n}(\xi)}.
\]
As a consequence,
\begin{equation}
  \label{xi-1-2}
  \begin{split}
  \abs{\xi_{n,1}-\xi_{n,2}} &\le \abs{a_{n}(\xi_{n,1})-a_{n}(\xi_{n,2})} + 
  \max_{\pm}\abs{\ph_{n}(\xi_{n,1}) \pm \ph_{n}(\xi_{n,2})}\\
  &\le \sup_{\lm\in D_{n}} \abs{\partial_{\lm} a_{n}(\lm)}\abs{\xi_{n,1}-\xi_{n,2}}
  + 2\sup_{\lm\in D_{n}} \abs{\ph_{n}(\lm)}.
  \end{split}
\end{equation}
Since
\[
  \dist(D_{n},\partial S_{n}) \ge 12n - 4n^{1/2} \ge 8n,
\]
one concludes from Cauchy's estimate and the estimate $2\n{q}_{s,\infty} \le n^{1/2-\abs{s}}$ following from~\eqref{cs-est} that
\[
  \sup_{\lm\in D_{n}} \abs{\partial_{\lm}a_{n}(\lm)}
   \le \frac{\sup_{\lm\in S_{n}} \abs{a_{n}(\lm)}}{\dist(D_{n},\partial S_{n})}
   \le \frac{\n{q}_{s,\infty}}{8n} \le \frac{1}{16}.
\]
Therefore, by \eqref{xi-1-2},
\[
  \abs{\xi_{n,1}-\xi_{n,2}}^{2} \le 6\sup_{\lm\in D_{n}}\abs{b_{n}(\lm)b_{-n}(\lm)}
\]
as claimed.~\qed
\end{proof}

\subsection{Proof of Theorem~\ref{thm:spec-fw} (i)}

Let $q\in \FL_{0,\C}^{w,s,\infty}$ with $-1/2< s \le 0$ and $w\in\Ms$. The eigenvalues of $L(q)$, when listed with lexicographic ordering, satisfy
\[
  \lm_{0}^{+}\lex \lm_{1}^{-}\lex \lm_{1}^{+}\lex \dotsb,\quad
  \text{and}\quad
  \lm_{n}^{\pm} = n^{2}\pi^{2} + n\ell_{n}^{2}.
\]
It follows from a standard counting argument that for $n\ge n_{s}$ with $n_{s}$ as in Corollary~\ref{Q-soln} that $\lm_{n}^{\pm}\in S_{n}$ and $\lm_{n}^{\pm} \notin S_{k}$ for any $k\neq n$. It then follows from Lemma~\ref{eval-char} and Lemma~\ref{Sn-roots} that $\setd{\xi_{n,1,},\xi_{n,2}} = \setd{\lm_{n}^{-},\lm_{n}^{+}}$ and hence $\gm_{n} = \lm_{n}^{+}-\lm_{n}^{-}$ satisfies
\[
  \abs{\gm_{n}} = \abs{\xi_{n,1}-\xi_{n,2}},\qquad \forall n\ge n_{s}.
\]

\begin{lem}
\label{gm-est}
If $q\in\FL_{0,\C}^{w,s,\infty}$ with $w\in\Ms$ and $-1/2 < s \le 0$, then for any $N\ge n_{s}$,
\[
  \n{T_{N}\gm(q)}_{w,s,\infty} \le
  4
  \n{T_{N}q}_{w,s,\infty}
   + \frac{16c_{s}}{N^{1/2-\abs{s}}}\n{q}_{w,s,\infty}^{2}.~\fish
\]
\end{lem}

\begin{proof}[Proof of Theorem~\ref{thm:spec-fw} (i).]
By Lemma~\ref{Sn-roots},
\begin{align*}
  \abs{\gm_{n}} &= \abs{\xi_{n,1}-\xi_{n,2}}
  \le \sqrt{3}\p[\Big]{\sup_{\lm\in S_{n}} \abs{b_{n}(\lm)} + \sup_{\lm\in S_{n}}\abs{b_{-n}(\lm)}}\\
  &\le
  \sqrt{3}
  \p*{
  \abs{q_{2n}} + \abs{q_{-2n}}
   + \sup_{\lm\in S_{n}} \abs{b_{n}(\lm) - q_{2n}}
          + \sup_{\lm\in S_{n}} \abs{b_{-n}(\lm) - q_{-2n}}
  }.
\end{align*}
It then follows from Lemma~\ref{Tn-est} and Lemma~\ref{coeff-est} that for $n\ge N$ with $N\defl n_{s}$
\[
  w_{2n}\lin{2n}^{s}\abs{\gm_{n}} \le
  \sqrt{3}
  \p*{
  w_{2n}\lin{2n}^{s}\abs{q_{2n}} + w_{2n}\lin{2n}^{s}\abs{q_{-2n}}
   + \frac{4c_{s}}{n^{1/2-\abs{s}}}\n{q}_{w,s,\infty}^{2}
  }.
\]
Thus, $(\gm_{n}(q))_{n\ge 1}\in\ell_{\C}^{w,s,\infty}(\N)$. As $n_{s}$ can be chosen locally uniformly in $q\in \FL_{0,\C}^{w,s,\infty}$, the map $\FL_{0,\C}^{w,s,\infty} \to \ell_{\C}^{w,s,\infty}(\N)$, $q\mapsto (\gm_{n}(q))_{n\ge 1}$ is locally bounded.~\qed
\end{proof}

\subsection{Jordan blocks of $L(q)$}

\label{ss:jordan}

To treat the Dirichlet problem, we develop the methods of \cite{Kappeler:CNzeErmy}, where the case $q\in \FL_{0,\C}^{s,p}$ with $-1/2\le s\le 0$ and $2\le p < \infty$ was considered, to the case with $-1/2< s \le 0$ and $p=\infty$.
If $q\in\FL_{0,\C}^{s,\infty}$ is not real valued, then the operator $L(q)$ might have complex eigenvalues and the geometric multiplicity of an eigenvalue could be less than its algebraic multiplicity. 

We choose $\check{n}_{s}\ge n_{s}$, where $n_{s}$ as in Corollary~\ref{Q-soln}, so that in addition
\begin{equation}
  \label{n-s-p}
  \begin{split}
  &\abs{\lm_{n}^{\pm}} \le (\check{n}_{s}-1)^{2}\pi^{2} + \check{n}_{s}/2,
  && \forall n < \check{n}_{s},\\
  &\abs{\lm_{n}^{\pm}-n^{2}\pi^{2}} \le n/2,
  && \forall n\ge \check{n}_{s},\\
  &\lm_{n}^{\pm}\text{ are 1-periodic [1-antiperiodic] if }n \text{ even [odd] }
  && \forall n\ge \check{n}_{s}.
  \end{split}
\end{equation}
Note that $\check{n}_{s}$ can be chosen uniformly on bounded subsets of $\FL_{0,\C}^{w,s,\infty}$ since $\FL_{0,\C}^{w,s,\infty}$ embeds compactly into $H_{0,\C}^{-1}$.
For $n\ge \check{n}_{s}$ we further let
\[
  E_{n} = \begin{cases}
  \Null(L-\lm_{n}^{+})\oplus\mathrm{Null}(L-\lm_{n}^{-}), &
  \lm_{n}^{+}\neq \lm_{n}^{-},\\
  \Null(L-\lm_{n}^{+})^{2}, & \lm_{n}^{+} = \lm_{n}^{-}.
  \end{cases}
\]
We need to estimate the coefficients of $L(q)\big|_{E_{n}}$ when represented with respect to an appropriate orthonormal basis of $E_{n}$. In the case where $\lm_{n}^{+} = \lm_{n}^{-}$ the matrix representation will be in Jordan normal form. By Lemma~\ref{regularity-En}, $E_{n}\subset\FL_{\star,\C}^{s+2,\infty}\opento L^{2} \defl L^{2}([0,2],\C)$. Denote by $f_{n}^{+}\in E_{n}$ an $L^{2}$-normalized eigenfunction corresponding to $\lm_{n}^{+}$ and by $\ph_{n}$ an $L^{2}$-normalized element in $E_{n}$ so that $\setd{f_{n}^{+},\ph_{n}}$ forms an $L^{2}$-orthonormal basis of $E_{n}$. 
Then the following lemma holds.

\begin{lem}
\label{En-basis}
Let $q\in\FL_{0,\C}^{s,\infty}$ with $-1/2< s\le 0$. Then there exists $n_{s}'\ge \check{n}_{s}$ -- with $\check{n_{s}}$ given by~\eqref{n-s-p} -- so that for any $n\ge n_{s}'$,
\[
  (L-\lm_{n}^{+})\ph_{n} = -\gm_{n}\ph_{n} + \eta_{n}f_{n}^{+},
\]
where $\eta_{n}\in \C$ satisfies the estimate
\[
  \abs{\eta_{n}} \le 16(\abs{\gm_{n}} + \abs{b_{n}(\lm_{n}^{+})} + \abs{b_{-n}(\lm_{n}^{+})}).
\]
The threshold $n_{s}'$ can be chosen locally uniformly in $q\in\FL_{0,\C}^{s,\infty}$.~\fish
\end{lem}

\begin{proof}
We begin by verifying the claimed formula for $(L-\lm_{n}^{+})\ph_{n}$ in the case where $\lm_{n}^{+}\neq \lm_{n}^{-}$. Let $f_{n}^{-}$ be an $L^{2}$-normalized eigenfunction corresponding to $\lm_{n}^{-}$. As $f_{n}^{-}\in E_{n}$ there exist $a,b\in\C$ with $\abs{a}^{2}+\abs{b}^{2} = 1$ and $b\neq0$ so that
\[
  f_{n}^{-} = a f_{n}^{+} + b\ph_{n}
  \quad\text{or}\quad
  \ph_{n} = \frac{1}{b}f_{n}^{-} - \frac{a}{b}f_{n}^{+}.
\]
Hence
\[
  L\ph_{n} = \frac{1}{b}\lm_{n}^{-}f_{n}^{-} - \frac{a}{b}\lm_{n}^{+}f_{n}^{+}.
\]
Substituting the expression for $f_{n}^{-}$ into the latter identity then leads to
\[
  (L-\lm_{n}^{+})\ph_{n} = (\lm_{n}^{-}-\lm_{n}^{+})\ph_{n} + \frac{a}{b}(\lm_{n}^{-}-\lm_{n}^{+})f_{n}^{+} = -\gm_{n}\ph_{n} + \eta_{n}f_{n}^{+}
\]
where $\eta_{n} = -\gm_{n}a/b$. In the case $\lm_{n}^{+}$ is a double eigenvalue of geometric multiplicity two, $\ph_{n}$ is an eigenfunction of $L$ and one has $\eta_{n} = 0$. Finally, in the case $\lm_{n}^{+}$  is a double eigenvalue of geometric multiplicity one, $(L-\lm_{n}^{+})\ph_{n}$ is in the eigenspace $E_{n}^{+}\subset E_{n}$ as claimed.

To prove the claimed estimate for $\eta_{n}$, we view $(L-\lm_{n}^{+})\ph_{n} = -\gm_{n}\ph_{n} + \eta_{n}f_{n}^{+}$ as a linear equation with inhomogeneous term $g = -\gm_{n}\ph_{n} + \eta_{n}f_{n}^{+}$. By identity~\eqref{B-inhom} one has
\[
  B_{n}P_{n}\ph_{n} = \gm_{n} P_{n}K_{n}\ph_{n} - \eta_{n}P_{n}K_{n}f_{n}^{+},
\]
where $K_{n} \equiv K_{n}(\lm_{n}^{+})$ and $B_{n} \equiv B_{n}(\lm_{n}^{+})$. To estimate $\eta_{n}$, take the $L^{2}$-inner product of the latter identity with $P_{n}f_{n}^{+}$ to get
\begin{equation}
  \label{eta-I-II}
  \eta_{n}\spii{P_{n}K_{n}f_{n}^{+},P_{n}f_{n}^{+}} = \gm_{n}I - II,
\end{equation}
where
\[
  I = \spii{P_{n}K_{n}\ph_{n}, P_{n}f_{n}^{+}},\qquad
  II = \spii{B_{n}P_{n}\ph_{n},P_{n}f_{n}^{+}}.
\]

We begin by estimating $\spii{P_{n}K_{n}f_{n}^{+},P_{n}f_{n}^{+}}$. Using that $K_{n} = \Id + T_{n}K_{n}$ one gets
\begin{align*}
  \spii{P_{n}K_{n}f_{n}^{+},P_{n}f_{n}^{+}}
   = \n{P_{n}f_{n}^{+}}_{L^2}^{2}
    + \spii{T_{n}K_{n}f_{n}^{+},P_{n}f_{n}^{+}},
\end{align*}
and by Cauchy-Schwarz
\[
  \abs{\spii{T_{n}K_{n}f_{n}^{+},P_{n}f_{n}^{+}}} \le
  \p[\bigg]{\sum_{m\in \setd{\pm n}} \abs{\spii{T_{n}K_{n}f_{n}^{+},e_{m}}}^{2}}^{1/2}\n{P_{n}f_{n}^{+}}_{L^{2}}.
\]
Note that $\n{P_{n}f_{n}^{+}}_{L^{2}} \le \n{f_{n}^{+}}_{L^{2}} = 1$.
Moreover, by Lemma~\ref{bn-est} one has
\[
  \abs{\spii{T_{n}K_{n}f_{n}^{+},e_{\pm n}}}
  \le 
  \begin{cases}
  \frac{\log{n}}{n}C_{s}\n{q}_{s,\infty}\n{K_{n}f_{n}^{+}}_{s,\infty;\pm n}, & s = 0,\\
  \frac{1}{n^{1-2\abs{s}}}C_{s}\n{q}_{s,\infty}\n{K_{n}f_{n}^{+}}_{s,\infty;\pm n}, & -1/2 < s < 0.
  \end{cases}
\]
By Corollary~\ref{Q-soln}, $\n{K_{n}}_{s,\infty;n} \le 2$ and as $L_{\C}^{2}[0,2] \opento \FL_{\star,\C}^{s,\infty}$, $\n{f_{n}^{+}}_{s,\infty;\pm n}\le \n{f_{n}^{+}}_{0,2;\pm n} =1$. Hence there exists $n_{s}'\ge \check{n}_{s}$ so that
\[
  \abs{\spii{T_{n}K_{n}f_{n}^{+},P_{n}f_{n}^{+}}} \le \frac{1}{8}.
\]
By increasing $n_{s}'$ if necessary, Lemma~\ref{Pn-est} below assures that $\n{P_{n}f_{n}^{+}}_{L^2} \ge 1/2$. Thus the left hand side of \eqref{eta-I-II} can be estimated as follows
\begin{equation}
  \label{eta-est-1}
  \abs*{\eta_{n}\spii{P_{n}K_{n}f_{n}^{+},P_{n}f_{n}^{+}}} \ge \abs{\eta_{n}}
  \p*{ \frac{1}{4}-\frac{1}{8} } = \frac{1}{8}\abs{\eta_{n}},
  \qquad
  \forall n\ge n_{s}'.
\end{equation}

Next let us estimate the term $I = \spii{P_{n}K_{n}\ph_{n}, P_{n}f_{n}^{+}}$ in \eqref{eta-I-II}. Using again $K_{n} = \Id + T_{n}K_{n}$ one sees that
\[
  I = \spii{P_{n}\ph_{n}, P_{n}f_{n}^{+}} + \spii{T_{n}K_{n}\ph_{n}, P_{n}f_{n}^{+}}.
\]
Clearly, $\abs{\spii{P_{n}\ph_{n}, P_{n}f_{n}^{+}}} \le \n{\ph_{n}}_{L^2}\n{f_{n}^{+}}_{L^2} \le 1$ and arguing as above for the second term, one then concludes that
\begin{equation}
  \label{eta-est-2}
  \abs{I} \le 1 + 1/8,\qquad \forall n\ge n_{s}'.
\end{equation}

Finally it remains to estimate $II = \spii{B_{n}P_{n}\ph_{n},P_{n}f_{n}^{+}}$. Using again $\n{\ph_{n}}_{L^{2}}  = \n{f_{n}^{+}}_{L^{2}} = 1$, we conclude from the matrix representation \eqref{Bn-mat} of $B_{n}$ that
\[
  \abs{\spii{B_{n}P_{n}\ph_{n},P_{n}f_{n}^{+}}} \le \n{B_{n}}\n{\ph_{n}}_{L^{2}}\n{f_{n}^{+}}_{L^{2}}
  \le \abs{\lm_{n}^{+}-n^{2}\pi^{2}-a_{n}} + \abs{b_{n}}+ \abs{b_{-n}}.
\]
Since $\det B_{n}(\lm_{n}^{+}) = 0$, one has
\[
  \abs{\lm_{n}^{+}-n^{2}\pi^{2}-a_{n}} = \abs{b_{n}b_{-n}}^{1/2} \le \frac{1}{2}(\abs{b_{n}} + \abs{b_{-n}}),
\]
and hence it follows that for all $n\ge n_{s}'$ that
\begin{equation}
  \label{eta-est-3}
  \abs{II} \le 2(\abs{b_{n}}+\abs{b_{-n}}).
\end{equation}
Combining \eqref{eta-est-1}-\eqref{eta-est-3} leads to the claimed estimate for $\eta_{n}$.~\qed
\end{proof}

It remains to prove the estimate of $P_{n}$ used in the proof of Lemma~\ref{En-basis}. To this end, we introduce for $n\ge n_{s}$ the Riesz projector $P_{n,q}\colon L^{2}\to E_{n}$ given by (see also Appendix~\ref{a:hill-op})
\[
  P_{n,q} = \frac{1}{2\pi \ii}\int_{\abs{\lm -n^{2}\pi^{2}} = n} (\lm - L(q))^{-1}\,\dlm.
\]

\begin{lem}
\label{Pn-est}
Let $q\in\FL_{0,\C}^{s,\infty}$ with $-1/2< s \le 0$. Then there exists $\tilde n_{s} \ge \check{n}_{s}$ -- with $\check{n_{s}}$ given by~\eqref{n-s-p} -- so that for any eigenfunction $f\in \FL_{\star,\C}^{s+2,p}$ of $L(q)$ corresponding to an eigenvalue $\lm\in S_{n}$ with $n\ge \tilde n_{s}$,
\[
  \n{P_{n}f}_{L^2}\ge \frac{1}{2}\n{f}_{L^2}.
\]
The threshold $\tilde n_{s}$ can be chosen locally uniformly for $q$.~\fish
\end{lem}

\begin{proof}
In Lemma~\ref{proj-bound} we show that as $n\to \infty$,
\[
  \n{P_{n,q}-P_{n}}_{L^{2}\to L^{\infty}} = o(1),
\]
locally uniformly in $q\in\FL_{0,\C}^{s,\infty}$.
Clearly, $P_{n,q}f = f$, hence
\[
  \n{P_{n}f}_{L^{2}} \ge \n{P_{n,q}f}_{L^{2}} - \n{(P_{n,q}-P_{n})f}_{L^{2}}
  \ge \p[\big]{1+o(1)}\n{f}_{L^{2}}.~\qed
\]
\end{proof}

\subsection{Proof of Theorem~\ref{thm:spec-fw} (ii)}

We begin with a brief outline of the proof of Theorem~\ref{thm:spec-fw} (ii). Let $q\in\FL_{0,\C}^{s,\infty}$ with $-1/2< s\le 0$. Since according to~\cite{Kappeler:2003vh} for any $q\in H_{0,\C}^{-1}$ the Dirichlet eigenvalues, when listed in lexicographical ordering and with their algebraic multiplicities, $\mu_{1}\lex \mu_{2}\lex \dotsb$, satisfy the asymptotics $\mu_{n} = n^{2}\pi^{2} + n\ell_{n}^{2}$, 
they are simple for $n\ge n_{\dir}$, where $n_{\dir}\ge 1$ can be chosen locally uniformly for $q\in H_{0,\C}^{-1}$. For any $n\ge n_{\dir}$ let $g_{n}$ be an $L^{2}$-normalized eigenfunction corresponding to $\mu_{n}$. Then
\[
  g_{n}\in H_{\dir,\C}^{1}\defl \setdef{g\in H^{1}([0,1],\C)}{g(0) = g(1) = 0}.
\]
Now let $q\in\FL_{0,\C}^{s,\infty}$ with $-1/2< s\le 0$. Increase $n_{s}'$ of Lemma~\ref{En-basis}, if necessary, so that $n_{s}'\ge n_{\dir}$ and denote by $E_{n}$ the two dimensional subspace introduced in Section~\ref{ss:jordan}.
We will choose an $L^{2}$-normalized function $\tilde G_{n}$ in $E_{n}$ so that its restriction $G_{n}$ to the interval $\Ic = [0,1]$ is in $H_{\dir,\C}^{1}$ and close to $g_{n}$. We then show that $\mu_{n}-\lm_{n}^{+}$ can be estimated in terms of $\spi{(L_{\dir}-\lm_{n}^{+})G_{n},G_{n}}$, where $\spi{f,g}$ denotes the $L^{2}$-inner product on $\Ic$, $\spi{f,g} = \int_{0}^{1} f(x)\ob{g(x)}\,\dx$. As by Lemma~\ref{En-basis}
\[
  (L-\lm_{n}^{+})\tilde G_{n} = O(\abs{\gm_{n}} + \abs{b_{n}(\lm_{n}^{+})} + \abs{b_{-n}(\lm_{n}^{+})}),
\]
the claimed estimates for $\mu_{n}-\tau_{n} = \mu_{n} - \lm_{n}^{+}+\gm_{n}/2$ then follow from the estimates of $\gm_{n}$ of Theorem~\ref{thm:spec-fw} (i) and the ones of $b_{n}-q_{2n}$, $b_{-n}-q_{-2n}$ of Lemma~\ref{coeff-est} (ii).

The function $\tilde G_{n}$ is defined as follows. Let $f_{n}^{+}$, $\ph_{n}$ be the $L^{2}$-orthonormal basis of $E_{n}$ chosen in Section~\ref{ss:jordan}. As $E_{n}\subset H_{\C}^{1}(\R/2\Z)$, its elements are continuous functions by the Sobolev embedding theorem. If $f_{n}^{+}(0) = 0$, then $f_{n}^{+}(1) = 0$ as $f_{n}^{+}$ is an eigenfunction of the $1$-periodic/antiperiodic eigenvalue $\lm_{n}^{+}$ of $L(q)$ and we set $\tilde G_{n} = f_{n}^{+}$.
If $f_{n}^{+}(0)\neq 0$, then we define $\tilde G_{n}(x) = r_{n}\bigl(\ph_{n}(0)f_{n}^{+}(x) - f_{n}^{+}(0)\ph_{n}(x)\bigr)$, where $r_{n} > 0$ is chosen in such a way that $\int_{0}^{1} \abs{\tilde G_{n}(x)}^{2}\,\dx = 1$. Then $\tilde G_{n}(0) = \tilde G_{n}(1) = 0$ and since $\tilde G_{n}$ is an element of $E_{n}$ its restriction $G_{n} \defl \tilde G_{n}\big|_{\Ic}$ is in $H_{\dir,\C}^{1}$.

Denote by $\Pi_{n,q}$ the Riesz projection, introduced in Appendix~\ref{a:hill-op},
\[
  \Pi_{n,q} \defl \frac{1}{2\pi \ii}\int_{\abs{\lm-n^{2}\pi^{2}} = n} (\lm-L_{\dir}(q))^{-1}\,\dlm.
\]
It has $\spanx(g_{n})$ as its range, hence there exists $\nu_{n}\in \C$ so that
\[
  \Pi_{n,q}G_{n} = \nu_{n}g_{n}.
\]

\begin{lem}
\label{gn-nu-est}
Let $q\in \FL_{0,\C}^{s,\infty}$ with $-1/2< s \le 0$. Then there exists $n_{s}''\ge n_{s}'$ with $n_{s}'$ as in Lemma~\ref{En-basis} so that for any $n\ge n_{s}''$
\begin{equation}
  \label{nu-eqn}
  \nu_{n}(\mu_{n}-\lm_{n}^{+})g_{n}
   = \bt_{n}\p*{ \eta_{n}\Pi_{n,q}(f_{n}^{+}\big|_{\Ic}) - \gm_{n}\Pi_{n,q}(\ph_{n}\big|_{\Ic}) },
\end{equation}
where $\bt_{n} \in \C$ with $\abs{\bt_{n}} \le 1$ and $\eta_{n}$ is the off-diagonal coefficient in the matrix representation of $(L-\lm_{n}^{+})\big|_{E_{n}}$ with respect to the basis $\setd{f_{n}^{+},\ph_{n}}$, introduced in Lemma~\ref{En-basis}, and
\begin{equation}
  \label{nu-est}
  1/2 \le \abs{\nu_{n}} \le 3/2.
\end{equation}
$n_{s}''$ can be chosen locally uniformly for $q\in\FL_{0,\C}^{s,\infty}$.~\fish
\end{lem}

\begin{proof}
Write $G_{n} = \nu_{n}g_{n} + h_{n}$, where $h_{n} = (\Id - \Pi_{n,q})G_{n}$. Then
\[
  (L_{\dir}-\lm_{n}^{+})G_{n} = \nu_{n}(\mu_{n}-\lm_{n}^{+})g_{n} + (L_{\dir}-\lm_{n}^{+})h_{n}.
\]
On the other hand, $G_{n} = \tilde G_{n}\big|_{\Ic}$, where $\tilde G_{n}\in E_{n}$ is given by $\tilde G_{n} = \al_{n}f_{n}^{+} + \bt_{n}\ph_{n}$ with $\al_{n}$, $\bt_{n}\in\C$ satisfying $\abs{\al_{n}}^{2} + \abs{\bt_{n}}^{2} = 1$ and $G_{n}\in H_{\dir,\C}^{1}$. Hence by Lemma~\ref{V-dir-V-per} and Lemma~\ref{En-basis}, for $n\ge n_{s}'$,
\[
  (L_{\dir}-\lm_{n}^{+})G_{n} = (L-\lm_{n}^{+})\tilde G_{n}\big|_{\Ic}
   = \bt_{n}(\eta_{n}f_{n}^{+} - \gm_{n}\ph_{n})\big|_{\Ic}.
\]
Combining the two identities and using that $\Pi_{n,q}h_{n} = 0$ and that $\Pi_{n,q}$ commutes with $(L_{\dir}-\lm_{n}^{+})$, one obtains, after projecting onto $\spanx(g_{n})$, identity~\eqref{nu-eqn}.

It remains to prove~\eqref{nu-est}. Taking the inner product of $\Pi_{n,q}G_{n} = \nu_{n} g_{n}$ with $g_{n}$ one gets
\[
  \nu_{n} = \nu_{n}\spi{g_{n},g_{n}} = \spi{\Pi_{n,q}G_{n},g_{n}}.
\]
Let $s_{n}(x) = \sqrt{2}\sin(n\pi x)$ and denote by $\Pi_{n} = \Pi_{n,0}$ the orthogonal projection onto $\spanx\setd{s_{n}}$. Recall that $P_{n,q}\colon L^{2}\to E_{n}$ is the Riesz projection onto $E_{n}$. In Lemma~\ref{proj-bound} we show that as $n\to \infty$,
\begin{equation}
  \label{proj-est-1}
  \n{\Pi_{n,q}-\Pi_{n}}_{L^{2}(\Ic)\to L^{\infty}(\Ic)},\;
  \n{P_{n,q}-P_{n}}_{L^{2}\to L^{\infty}} = o(1),
\end{equation}
locally uniformly in $q\in\FL_{0,\C}^{s,\infty}$.
Thus using $\Pi_{n}G_{n} = \Pi_{n}(P_{n}\tilde G_{n})\big|_{\Ic}$ and recalling that $\n{G_{n}}_{L^{2}(\Ic)}^{2} = \n{g_{n}}_{L^{2}(\Ic)}^{2} = 1$ we obtain
\[
  v_{n}
   = \spi{\Pi_{n}G_{n},g_{n}} + \spi{(\Pi_{n,q}-\Pi_{n})G_{n},g_{n}}
   = \spi{P_{n}\tilde G_{n},\Pi_{n}g_{n}} + o(1).
\]
Moreover, it follows from \eqref{proj-est-1} that uniformly in $0\le x \le 1$
\[
  \Pi_{n}g_{n}(x) = \e^{\ii \phi_{n}}s_{n}(x) + o(1),\qquad n\to \infty,
\]
with some real $\phi_{n}$. Similarly, again by \eqref{proj-est-1}, uniformly in $0\le x \le 2$
\[
  P_{n}\tilde G_{n}(x) = a_{n}e_{n}(x) + b_{n}e_{-n}(x) + o(1),\qquad n\to \infty,
\]
where, since $\n{G_{n}}_{L^{2}(\Ic)} = 1$ and $G_{n}(0) = 0$, the coefficients $a_{n}$ and $b_{n}$ can be chosen so that
\[
  \abs{a_{n}}^{2} + \abs{b_{n}}^{2} = 1,\quad a_{n} + b_{n} = 0.
\]
That is $P_{n}\tilde G_{n}(x) = \e^{\ii \psi_{n}}s_{n}(x) + o(1)$ with some real $\psi_{n}$ and hence
\[
  \spi{P_{n}\tilde G_{n},\Pi_{n}g_{n}}
   = \e^{\ii\psi_{n}-\ii\phi_{n}}\spi{s_{n},s_{n}} + o(1)
   = \e^{\ii\psi_{n}-\ii\phi_{n}} + o(1),\qquad n\to \infty.
\]
From this we conclude
\[
  \abs{\nu_{n}} = 1 + o(1),\qquad n\to\infty.
\]
Therefore, $1/2 \le \abs{\nu_{n}} \le 3/2$ for all $n\ge n_{s}''$ provided $n_{s}'' \ge n_{s}'$ is sufficiently large.

Going through the arguments of the proof one verifies that $n_{s}''$ can be chosen locally uniformly in $q$.~\qed
\end{proof}

Lemma~\ref{gn-nu-est} allows to complete the proof of Theorem~\ref{thm:spec-fw} (ii).

\begin{proof}[Proof of Theorem~\ref{thm:spec-fw} (ii).]
Take the inner product of~\eqref{nu-eqn} with $g_{n}$ and use that $\abs{\nu_{n}} \ge 1/2$ by Lemma~\ref{gn-nu-est} to conclude that
\begin{equation}
  \label{mu-est-1}
  \frac{1}{2}\abs{\mu_{n}-\lm_{n}^{+}} \le \abs{\bt_{n}}
  \p*{ \abs{\eta_{n}} \spi{\Pi_{n,q}(f_{n}^{+}\big|_{\Ic}),g_{n}}
  + \abs{\gm_{n}}\abs{\spi{\Pi_{n,q}(\ph_{n}\big|_{\Ic}),g_{n}}} }.
\end{equation}
Recall that $\abs{\bt_{n}} \le 1$ and note that for any $f,g\in L_{\C}^{2}(\Ic)$
\[
  \abs{\spi{\Pi_{n,q}f,g}} \le 
  \abs{\spi{\Pi_{n}f,g}} +
  \abs{\spi{(\Pi_{n,q}-\Pi_{n})f,g}} \le 
  (1+o(1))\n{f}_{L^{2}(\Ic)}\n{g}_{L^{2}(\Ic)},
\]
where for the latter inequality we used that by Lemma~\ref{proj-bound} (ii), $\n{\Pi_{n,q}-\Pi_{n}}_{L^{2}(\Ic)\to L^{2}(\Ic)} = o(1)$ as $n\to \infty$.
Since $\n{f_{n}^{+}}_{L^{2}(\Ic)} = \n{\ph_{n}}_{L^{2}(\Ic)} = 1$ and $\n{g_{n}}_{L^{2}(\Ic)} = 1$, \eqref{mu-est-1} implies that
\[
  \abs{\mu_{n}-\lm_{n}^{+}} \le (2+o(1))(\abs{\eta_{n}}+\abs{\gm_{n}})
\]
yielding with Lemma~\ref{En-basis} the estimate
\[
  \abs{\mu_{n}-\tau_{n}} \le  (3+o(1))\abs{\gm_{n}}
  + (32+o(1))(\abs{\gm_{n}} + \abs{b_{n}(\lm_{n}^{+})}+
  \abs{b_{-n}(\lm_{n}^{+})}).
\]
By Theorem~\ref{thm:spec-fw} (i) and Lemma~\ref{coeff-est} (ii) it then follows that $(\tau_{n}-\mu_{n})_{n\ge 1}\in \ell_{\C}^{w,s,\infty}(\N)$. Going through the arguments of the proof one verifies that the map 
$\FL_{0,\C}^{w,s,\infty}\to \ell_{\C}^{w,s,\infty}(\N)$, 
$q\mapsto (\tau_{n}-\mu_{n})_{n\ge 1}$ is locally bounded.~\qed
\end{proof}

\subsection{Adapted Fourier Coefficients}

The bounds of the operator norm $\n{T_{n}}_{w,s,\infty;n}$ and the coefficients $a_{n}$ and $b_{\pm n}$ of $P_{n}K_{n}(\lm)V$, $\lm\in S_{n}$, obtained in Lemma~\ref{Tn-est} and Lemma~\ref{coeff-est}, respectively, are uniform in $\lm\in S_{n}$ and in $q$ on bounded subsets of $\FL_{0,\C}^{w,s,\infty}$. In addition, they are also uniform with respect to certain ranges of $p$ and the weight $w$. To give a precise statement we introduce the balls
\[
  B_{m}^{w,s,\infty} \defl \setdef{q\in \FL_{0,\C}^{w,s,\infty}}{\n{q}_{w,s,\infty}\le m},\qquad
  B_{m}^{s,\infty} \defl \setdef{q\in \FL_{0,\C}^{s,\infty}}{\n{q}_{s,\infty}\le m}.
\]
Then according to Lemma~\ref{Tn-est}, given $m> 0$ and $-1/2 < s \le 0$, one can choose $N_{m,s}$ so that
\begin{equation}
  \label{cs2-estimate}
  \frac{16c_{s}'m}{n^{1/2-\abs{s}}} \le 1/2,\qquad n\ge N_{m,s},
\end{equation}
where $c_{s}'\ge c_{s}\ge 1$ is chosen as in Lemma~\ref{bn-est}. This estimate implies that
\[
  \n{T_{n}(\lm)}_{w,s,\infty;n} \le 1/2,\qquad \forall \lm\in S_{n},\quad w\in\Ms,\quad q\in B_{2m}^{w,s,\infty}.
\]

\begin{lem}
\label{an-bn-est}
Let $-1/2 < s \le 0$ and $m\ge 1$. For $n\ge N_{m,s}$ with $N_{m,s}$ given as in~\eqref{cs2-estimate}, the coefficients $a_{n}$ and $b_{\pm n}$ are analytic functions on $S_{n}\times B_{m}^{s,\infty}$. Moreover, their restrictions to $S_{n}\times B_{2m}^{w,s,\infty}$ for any $w\in\Ms$ satisfy
\begin{align*}
  \text{(i)}\quad &
  \abs{a_{n}}_{S_{n}\times B_{2m}^{w,s,\infty}}
  \le \frac{8c_{s}m^{2}}{n^{1/2-\abs{s}}} \le m/4.\\
  \text{(ii)}\quad &
  w_{2n}\lin{2n}^{s}\abs{b_{\pm n}-q_{\pm 2n}}_{S_{n}\times B_{2m}^{w,s,\infty}}
  \le
  \frac{\log \w{n}}{n^{1-\abs{s}}} 8c_{s}'m^{2}
  \le 
  \frac{\log \w{n}}{n^{1/2}}m/4.~\fish
\end{align*}
\end{lem}

\begin{proof}
The claimed analyticity follows from the representations~\eqref{an-bn-expansions} of $a_{n}$ and $b_{\pm n}$ and the bounds from Lemma~\ref{Tn-est}, Lemma~\ref{coeff-est}, and Lemma~\ref{bn-est}.~\qed
\end{proof}

\begin{lem}
\label{alpha-n}
Let $-1/2 < s \le 0$ and $m\ge 1$. For each $n\ge N_{m,s}$ with $N_{m,s}$ given as in~\eqref{cs2-estimate}, there exists a unique real analytic function
\[
  \al_{n}\colon B_{2m}^{s,\infty}\to \C,\qquad 
  \abs{\al_{n}-n^{2}\pi^{2}}_{B_{2m}^{s,\infty}}
   \le \frac{8c_{s}m^{2}}{n^{1/2-\abs{s}}}
   \le m/4,
\]
such that $\lm-n^{2}\pi^{2}-a_{n}(\lm,\cdot)|_{\lm=\al_{n}}\equiv 0$ identically on $B_{2m}^{s,\infty}$.~\fish
\end{lem}

\begin{proof}
We follow the proof of \cite[Lemma~5]{Poschel:2011iua}.
Let $E$ denote the space of analytic functions $\al\colon B_{2m}^{s,\infty}\to \C$ with $\abs{\al-n^{2}\pi^{2}}_{B_{2m}^{s,\infty}} \le \frac{8c_{s}m^{2}}{n^{1/2-\abs{s}}}$ equipped with the usual metric induced by the topology of uniform convergence. This space is complete -- cf. \cite[Theorem~A.4]{Grebert:2014iq}.
Fix any $n\ge N_{m,s}$ and consider on $E$ the fixed point problem for the operator $\Lm_{n}$,
\[
  \Lm_{n}\al \defl n^{2}\pi^{2} + a_{n}(\al,\cdot).
\]
By~\eqref{cs2-estimate}, each such function satisfies
\[
  \abs{\al-n^{2}\pi^{2}}_{B_{2m}^{s,\infty}}  \le 2m < 4n^{1/2},
\]
and hence maps the ball $B_{2m}^{s,\infty}$ into the disc $D_{n} = \setd{\abs{\lm-n^{2}\pi^{2}} \le 4n^{1/2}}\subset S_{n}$. Therefore, by Lemma~\ref{an-bn-est}
\[
  \abs{\Lm_{n}\al-n^{2}\pi^{2}}_{B_{2m}^{s,\infty}}
   \le \abs{a_{n}}_{S_{n}\times B_{2m}^{s,\infty}} 
   \le \frac{8c_{s}m^{2}}{n^{1/2-\abs{s}}},
\]
meaning that $\Lm_{n}$ maps $E$ into $E$. Moreover, $\Lm_{n}$ contracts by a factor $1/4$ by Cauchy's estimate,
\[
  \abs{\partial_{\lm}a_{n}}_{D_{n}\times B_{2m}^{s,\infty}} \le \frac{\abs{a_{n}}_{S_{n}\times B_{2m}^{s,\infty}}}{\dist(D_{n},\partial S_{n})}
  \le \frac{1}{12n-4n^{1/2}}\frac{8c_{s}m^{2}}{n^{1/2-\abs{s}}} \le \frac{1}{4}.
\]
Hence, we find a unique fixed point $\al_{n} = \Lm_{n}\al_{n}$ with the properties as claimed.~\qed
\end{proof}

To simplify notation define $\al_{-n} \defl \al_{n}$ for $n\ge 1$. For any given $m\ge 1$, define the map $\Fm^{(m)}$ on $B_{2m}^{s,\infty}$ by
\[
  \Fm^{(m)}(q) = \sum_{0\neq \abs{n} < M_{m,s}}q_{2n}e_{2n} + \sum_{\abs{n} \ge M_{m,s}} b_{n}(\al_{n}(q),q)e_{2n},
\]
where $M_{m,s} \ge N_{m,s}$ is chosen such that
\begin{equation}
  \label{M-choice}
  \sup_{n\ge M_{m,s}} \frac{8c_{s}'}{n^{1/2-\abs{s}}} \le \frac{1}{16m}.
\end{equation}
Thus, for $n\ge M_{m,s}$ the Fourier coefficients of the 1-periodic function $r = \Fm^{(m)}(q)$ are $r_{2n} = b_{n}(\al_{n}(q))$, $r_{-2n} = b_{-n}(\al_{-n}(q))$, and
\[
  B_{n}(\al_{n}(q),q) = 
  \begin{pmatrix}
  0 & -r_{2n}\\
  -r_{-2n} & 0
  \end{pmatrix}.
\]
These new Fourier coefficients are adapted to the lengths of the corresponding spectral gaps, whence we call $\Fm^{(m)}$ the \emph{adapted Fourier coefficient map} on $B_{m}^{s,\infty}$.

\begin{prop}
\label{Phi-diffeo}
For $-1/2 < s \le 0$ and $m\ge 1$, $\Fm^{(m)}$ maps $B_{m}^{s,\infty}$ into $\FL_{0,\C}^{s,\infty}$. Further, for every $w\in\Ms$, its restriction to $B_{m}^{w,s,\infty}$ is a real analytic diffeomorphism
\[
  \Fm^{(m)}\big|_{B_{m}^{w,s,\infty}}\colon B_{m}^{w,s,\infty} \to \Fm^{(m)}(B_{m}^{w,s,\infty})\subset\FL_{0,\C}^{w,s,\infty}
\]
such that
\begin{equation}
  \label{dFm-est}
  \sup_{q\in B_{m}^{w,s,\infty}}\n{\ddd_{q}\Fm^{(m)}-\Id}_{w,s,\infty}\le 1/16,
\end{equation}
and $B_{m/2}^{w,s,p}\subset \Fm^{(m)}(B_{m}^{w,s,p})$. Moreover,
\begin{equation}
  \label{Om-2side-est}
  \frac{1}{2}\n{q}_{w,s,\infty} \le \n{\Fm^{(m)}(q)}_{w,s,\infty} \le 2\n{q}_{w,s,\infty},\qquad q\in B_{m}^{w,s,\infty}.~\fish
\end{equation}
\end{prop}

\begin{proof}
Since $\al_{n}$ maps $B_{2m}^{s,\infty}$ into $S_{n}$ for $n\ge N_{m,s}$, each coefficient $b_{n}(\al_{n}(q),q)$ is well defined for $q\in B_{2m}^{s,\infty}$, and by Lemma~\ref{an-bn-est}
\[
  w_{2n}\lin{2n}^{s}\abs{b_{n}(\al_{n})-q_{2n}}_{B_{2m}^{s,\infty}}
  \le
  w_{2n}\lin{2n}^{s}\abs{b_{n}-q_{2n}}_{S_{n}\times B_{2m}^{s,\infty}}
  \le 
  \frac{8c_{s}m^{2}}{n^{1/2-\abs{s}}}.
\]
Hence the map $\Fm^{(m)}$ is defined on $B_{2m}^{s,\infty}$ and
\begin{align*}
  \sup_{q\in B_{2m}^{w,s,\infty}}\n{\Fm^{(m)}(q)-q}_{w,s,\infty}
  &=
  \sup_{q\in B_{2m}^{w,s,\infty}}
  \sup_{\abs{n}\ge M_{m,s}}
  w_{2n}\lin{2n}^{s}\abs{b_{n}(\al_{n})-q_{2n}}_{B_{2m}^{w,s,\infty}} \\
  &\le
  8c_{s}m^{2}
  \sup_{n\ge M_{m,s}} 
  \frac{1}{n^{1/2-\abs{s}}}\\
  &\le
  \frac{m^{2}}{16m}
  \le \frac{m}{16},
\end{align*}
by our choice~\eqref{M-choice} of $M_{m,s}$.
Consequently, the inverse function theorem (Lemma~\ref{ift}) applies proving that $\Fm^{(m)}$ is a diffeomorphism onto its image which covers $B_{m/2}^{w,s,p}$. If $q$ is real valued, then $\al_{n}(q)$ is real and hence by~\eqref{bn-symmetry} $b_{-n}(\al_{n}(q)) = -\ob{b_{n}(\al_{n}(q))}$ implying that $\Fm^{(m)}(q)$ is real valued as well. Altogether we thus have proved that $\Fm^{(m)}$ is real analytic.

Finally, we note that by Cauchy's estimate
\[
  \sup_{q\in B_{m}^{w,s,\infty}}\n{\ddd_{q}\Fm^{(m)}-\Id}_{w,s,\infty} \le 
  \frac{\sup_{q\in B_{2m}^{w,s,\infty}}\n{\Fm^{(m)}(q)-q}_{w,s,\infty}}{m}
  \le \frac{1}{16},
\]
hence in view of the mean value theorem for any $q_{0},q_{1}\in B_{m}^{w,s,p}$
\[
  \Fm^{(m)}(q_{1}) - \Fm^{(m)}(q_{0})
  =
  \p*{1 + \int_{0}^{1} \p*{\ddd_{(1-t)q_{0}+tq_{1}}\Fm^{(m)}-\Id}\,\dt}(q_{1}-q_{0}),
\]
we find that
\[
  \frac{1}{2}\n{q_{1}-q_{0}}_{w,s,\infty}\le \n{\Fm^{(m)}(q_{1}) - \Fm^{(m)}(q_{0})}_{w,s,\infty}
  \le 2\n{q_{1}-q_{0}}_{w,s,\infty}.~\qed
\]
\end{proof}

\subsection{Proof of Theorem~\ref{inv}}

\begin{prop}
\label{abstract-diffeo}
Let $-1/2 < s \le 0$, $m\ge 1$, and $w\in\Ms$. If $q\in B_{m}^{s,\infty}$ and
\[
  \Fm^{(m)}(q) \in B_{m/2}^{w,s,\infty},
\]
then $q\in B_{m}^{w,s,\infty}\subset \FL_{0,\C}^{w,s,\infty}$.~\fish
\end{prop}

\begin{proof}
By Proposition~\ref{Phi-diffeo}, the map $\Fm^{(m)}$ is defined on $B_{m}^{s,\infty}$ and a real analytic diffeomorphism onto its image; for $w\in\Ms$, the restriction of $\Fm^{(m)}$ to $B_{m}^{w,s,\infty}\subset B_{m}^{s,\infty}\cap \FL_{0,\C}^{w,s,\infty}$ is again a real analytic diffemorphism onto its image and by Lemma~\ref{Phi-diffeo} this image contains $B_{m/2}^{w,s,\infty}$. Thus, if $\Fm^{(m)}$ maps $q\in B_{m}^{s,\infty}$ to
\[
  r = \Fm^{(m)}(q)\in B_{m/2}^{w,s,\infty},
\] 
then we must have
\[
  q = (\Fm^{(m)})^{-1}\big|_{B_{m/2}^{w,s,\infty}}\,(r)\in B_{m}^{w,s,\infty}\subset \FL_{0,\C}^{w,s,\infty}.~\qed
\]
\end{proof}

To proceed, we want to bound the Fourier coefficients of $r = \Fm^{(m)}(q)$ in terms of the gap lengths of $q$. 

\begin{lem}
\label{root-est-poschel}
Let $-1/2 < s \le 0$, $m\ge 1$, and suppose that $q\in B_{m}^{s,\infty}$, $r = \Fm^{(m)}(q)$, and $n\ge M_{m,s}$ with $M_{m,s}$ given as in~\eqref{M-choice}. If
\[
  r_{-2n}\neq 0,\quad\text{and}\quad \frac{1}{9} \le \abs*{\frac{r_{2n}}{r_{-2n}}} \le 9,
\]
then
\[
  \abs{r_{2n}r_{-2n}} \le \abs{\gm_{n}(q)}^{2} \le 9\abs{r_{2n}r_{-2n}}.~\fish
\]
\end{lem}

\begin{proof}
We follow the proof of \cite[Lemma~10]{Poschel:2011iua}. To begin, we write $\det B_n(\lm) =
g_+(\lm)g_-(\lm)$ with
\[
  g_\pm(\lm) \defl \lm - n^{2}\pi^{2} - a_n(\lm) \mp \ph_n(\lm),\qquad 
  \ph_n(\lm) = \sqrt{b_n(\lm) b_{-n}(\lm)}.
\]
The assumption on $r_{\pm 2n}$ implies that $g_\pm$ are continuous, even analytic, functions of $\lm$. Indeed, recall that $r_{\pm 2n} =b_{\pm n}(\al_n)$, thus
\[
  \ph_n(\al_n) = \sqrt{b_n(\al_n)b_{-n}(\al_n)} \neq 0,\qquad 
  \rho_n \defl \abs{\ph_n(\al_{n})} > 0,
\]
so we may choose $\ph_n(\lm)$ as a fixed branch of the square root locally around $\lm=\al_n$. To obtain an estimate of the domain of analyticity, we consider the disc $D_n^\o\defl\setdef{\lm}{\abs{\lm-\al_n}\le 2\rho_n}$.
Since by assumption $n\ge M_{m,s}$ it follows from~\eqref{M-choice} together with $c_{s}\ge 1$ that
\[
  \frac{m}{4} \le \frac{n^{1/2-\abs{s}}}{128},\qquad \forall n\ge M_{m,s}.
\]
Lemma~\ref{an-bn-est} then yields
\begin{equation}
  \label{an-temp-est}
  \abs{a_{n}}_{S_{n}} \le \frac{m}{4} \le \frac{n^{1/2-\abs{s}}}{128}.
\end{equation}
To estimate $\abs{b_{n}}_{S_{n}}$ note that
\[
  \abs{b_{n}}_{S_{n}} \le \w{2n}^{\abs{s}}(\w{2n}^{s}\abs{b_{n}-q_{2n}}_{S_{n}} + \w{2n}^{s}\abs{q_{2n}}).
\]
Since $q\in B_{m}^{s,\infty}$, $\w{2n}^{s}\abs{q_{2n}} \le m$ and hence again by Lemma~\ref{an-bn-est} one has
\begin{equation}
  \label{bn-temp-est}
  \abs{b_{n}}_{S_{n}} \le 2n^{\abs{s}}(m/2+m) \le 4n^{\abs{s}}m \le \frac{n^{1/2}}{8}.
\end{equation}
Cauchy's estimate and definition~\eqref{Sn} of $S_{n}$ then gives
\begin{equation}
  \label{dbn-est}
  \abs{\partial_\lm b_{\pm n}}_{D_n^\o}
  \le \frac{\abs{b_{\pm n}}_{S_n}}{\dist(D_n^\o,\partial S_n)}
  \le \frac{n^{1/2}/8}{12n - \abs{n^{2}\pi^{2}-\al_{n}} - 2\rho_n}
  \le \frac{1}{88},
\end{equation}
where we used that by Lemma~\ref{alpha-n}, $\abs{\al_n-n^{2}\pi^{2}}\le m/4 \le n^{1/2-\abs{s}}/128$ and by~\eqref{bn-temp-est}, $\rho_{n} \le n^{1/2}/8$. Note that by~\eqref{an-temp-est}, the same estimate holds for $\partial_{\lm}a_{n}$,
\begin{equation}
  \label{dan-est}
  \abs{\partial_\lm a_n}_{D_{n}^{\o}} \le 1/88
\end{equation}
Thus by the mean value theorem, for any $\lm\in D_{n}^{\o}$,
\begin{equation}
  \label{bn-lm-bn-aln}
  \abs{b_{\pm n}(\lm)-b_{\pm n}(\al_n)}_{D_n^\o}
   \le
  \abs{\partial_\lm b_{\pm n}}_{D_n^\o}2\rho_n \le \frac{1}{44}\rho_n,
\end{equation}
implying that $\ph_{n}(\lm)^{2}$ is bounded away from zero for $\lm \in D_{n}^{\o}$. Hence $\ph_{n}(\lm)$ is analytic for $\lm\in D_{n}^{\o}$.

By Lemma~\ref{Sn-roots}, $\det B_{n}(\lm) = g_{+}(\lm)g_{-}(\lm)$ has precisely two roots in $S_{n}$ which both are contained in $D_{n}\subset S_{n}$. To estimate the location of these roots, we approximate $g_{\pm}(\lm)$ by $h_{\pm}(\lm)$ defined by
\begin{align*}
  h_{+}(\lm) &= \lm-n^{2}\pi^{2}-a_{n}(\al_{n}) - \ph_{n}(\al_{n}),\\
  h_{-}(\lm) &= \lm-n^{2}\pi^{2}-a_{n}(\al_{n}) + \ph_{n}(\al_{n}).
\end{align*}
Since $\al_{n}-n^{2}\pi^{2}-a_{n}(\al_{n}) = 0$, one has
\[
  h_{+}(\lm) = \lm-\al_{n}-\ph_{n}(\al_{n}),\qquad
  h_{-}(\lm) = \lm-\al_{n}+\ph_{n}(\al_{n}).
\]
Clearly, $h_{+}(\lm)$ and $h_{-}(\lm)$ each have precisely one zero $\lm_{+} = \al_{n}+\ph_{n}(\al_{n})$ and $\lm_{-} = \al_{n} - \ph_{n}(\al_{n})$, respectively.

We want to compare $h_+$ and $g_+$ on the disc
\[
  D_n^+\defl \setdef{\lm}{\abs{\lm-(\al_n + \ph_{n}(\al_{n}))}< \rho_n/2} \subset D_n^\o.
\]
Since $h_+(\al_n + \ph_{n}(\al_{n})) = 0$, we have
\[
  \abs{h_+}_{\partial D_n^+}
   = \abs{h_+(\lm)-h_+(\al_n + \ph_{n}(\al_{n}))}_{\partial D_n^+}
   = \frac{\rho_n}{2}.
\]
In the sequel we show that
\begin{equation}
  \label{sgn-estimate}
  \abs{\partial_\lm \ph_n}_{D_n+} \le \frac{4}{88},
\end{equation}
yielding together with~\eqref{dan-est}
\begin{align*}
 \abs{h_+-g_+}_{D_n^+}
 &\le \abs{a_n(\al_n) - a_n(\lm)}_{D_n^+} + \abs{\ph_n(\al_n)-\ph_n(\lm)}_{D_n^+}\\
 &\le \p*{\abs{\partial_\lm a_n}_{D_n^+} + \abs{\partial_\lm \ph_n}_{D_n^+}} 2\rho_n
    < \frac{\rho_n}{2} = \abs{h_+}_{\partial D_n^+}.
\end{align*}
Thus, it follows from Rouche's theorem that $g_+$ has a single root contained in $D_n^+$. In a similar fashion, we find that $g_{-}$ has a single root contained in $D_n^-\defl\setdef{\lm}{\abs{\lm-(\al_n-\ph_{n}(\al_{n}))}<\rho_n/2}$. Since the roots of $g_\pm(\lm)$ are roots of $\det B_n(\lm)$, they have to coincide with $\lm_n^\pm$ and hence
\[
  \rho_n \le \abs{\lm_n^+-\lm_n^-} \le 3\rho_n,
\]
which is the claim.

It remains to show the estimate~\eqref{sgn-estimate} for $\partial_\lm \ph_n$ on $D_n^+$.
Note that $\ob{D_{n}^{+}}\subset D_{n}^{\o}$ and write 
\[
  \abs{b_n(\lm)} = \abs{b_n(\lm)b_{-n}(\lm)}^{1/2}\abs*{\frac{b_n(\lm)}{b_{-n}(\lm)}}^{1/2}.
\]
By the assumption of this lemma, $\frac{1}{9} \le \frac{\abs{b_{n}(\al_{n})}}{\abs{b_{-n}(\al_{n})}}\le 9$ and by the estimate~\eqref{bn-lm-bn-aln},
\[
  \frac{1}{3}\rho_n \le \abs{b_n(\al_n)} \le 3\rho_n,
  \qquad
  \abs{b_n(\lm)-b_n(\al_n)}_{D_n^\o} \le \frac{\rho_{n}}{44}.
\]
Thus, by the triangle inequality we obtain
\[
  \frac{41}{132}\rho_n
    = \p*{\frac{1}{3}-\frac{1}{44}}\rho_n
  \le \abs{b_n^+}_{D_n^\o} 
  \le \p*{3+\frac{1}{44}}\rho_n
    = \frac{133}{44}\rho_n.
\]
Treating $b_{-n}$ in an analogous way, we arrive at
\[
  \abs*{\frac{b_n^+}{b_n^-}}_{D_n^\o},
  \abs*{\frac{b_n^-}{b_n^+}}_{D_n^\o}
  \le 10,
\]
which in view of~\eqref{dbn-est} finally yields the desired estimate~\eqref{sgn-estimate},
\begin{align*}
  \abs{\partial_\lm \ph_n}_{D_n^\o}
  \le
  \frac{\abs{\partial_\lm b_n^+}_{D_n^\o}}{2}
  \abs*{\frac{b_n^-}{b_n^+}}_{D_n^\o}^{1/2}
  +
  \frac{\abs{\partial_\lm b_n^-}_{D_n^\o}}{2}
  \abs*{\frac{b_n^+}{b_n^-}}_{D_n^\o}^{1/2}
  \le \frac{4}{88}.~\qed
\end{align*}

\end{proof}

\begin{lem}
\label{iso-bounded}
If $q_{0}\in H_{0}^{-1}$ with gap lengths $\gm(q_{0}) \in \ell^{s,\infty}$, $-1/2 < s\le 0$, then $\Iso(q_{0})$ is a $\n{\cdot}_{s,\infty}$-norm bounded subset of $\FL_{0}^{s,\infty}$.
In particular, $q_{0}$ is in $\FL_{0}^{s,\infty}$.~\fish
\end{lem}

\begin{proof}
Suppose $q_{0}$ is a real valued potential in $H_{0}^{-1}$ with gap lengths $\gm(q_{0}) \in \ell^{s,\infty}$ for some $-1/2 < s \le 0$.
We can choose $-1/2 < \sg < s$ and $2\le p < \infty$ so that $(s-\sg)p > 1$ and hence $\ell^{s,\infty}\opento \ell^{\sg,p}$. Consequently, by \cite[Corollary~3]{Kappeler:CNzeErmy} we have $q_{0}\in \FL^{\sg,p}$. Moreover, by \cite[Corollary~4]{Kappeler:CNzeErmy} the isospectral set $\Iso(q_{0})$ is compact  in $\FL^{\sg,p}$, hence there exists $R > 0$ so that $\Iso(q_{0})$ is contained in the ball $B_{R}^{\sg,p}$.
To prove that $\Iso(q_{0})$ is a bounded subset of $\FL^{s,\infty}$, we choose
\begin{equation}
  \label{m-choice}
  m = 4(R + \n{\gm(q_{0})}_{s,\infty}).
\end{equation}
Further, let $w_{n} = \w{n}^{-\sg+s}$, $n\in\Z$. Then $w\in\Ms$, $\ell^{w,\sg,\infty} = \ell^{s,\infty}$, and $\gm(q_{0}) \in \ell^{w,\sg,\infty}$, while for any $q\in \Iso(q_{0})$ we have
\[
  q\in B_{m}^{\sg,\infty}.
\]
The map $\Fm^{(m)}$ is well defined on $B_{m}^{\sg,\infty}$ and
\[
  r\equiv r(q) = \Fm^{(m)}(q) \in \FL_{0}^{\sg,\infty}.
\]
Since $r$ is real valued, we have $r_{-n} = \ob{r_{n}}$ for all $n\in\Z$. Suppose $\abs{n}\ge M_{m,s}$, then it follows from Lemma~\ref{root-est-poschel} that for any $q\in \Iso(q_{0})$ with $\abs{r_{n}}\neq 0$ that
\[
  \abs{r_{n}} = \abs{r_{n}r_{-n}}^{1/2} \le \abs{\gm_{n}(q)} = \abs{\gm_{n}(q_{0})}.
\]
The same estimate holds true when $\abs{r_{n}} = 0$. In particular, it follows that $r\in \FL_{0}^{w,\sg,\infty}$.
To satisfy the smallness assumption of Proposition~\ref{abstract-diffeo} for $\n{r}_{w,\sg,\infty}$, we modify the weight $w$: let $w^{\ep}$ be the weight defined by $w^{\ep}_{n} = \min(w_{n},\e^{\ep \abs{n}})$, $n\in\Z$. Note that $w_{-n}^{\ep}=w_{n}^{\ep}$, $w_{n}^{\ep}\ge 1$, and $w_{\abs{n}} \le w_{\abs{n}+1}$ for any $n\in\Z$. For $\ep > 0$ sufficiently small, one verifies that $\log w_{n+m}^{\ep} \le \log w_{n}^{ep} + \log w_{m}^{\ep}$ for any $n,m\in\Z$. Thus for $\ep > 0$ sufficiently small, $w^{\ep}$ is submultiplicative and therefore $w^{\ep}\in\Ms$ -- see \cite[Lemma~9]{Poschel:2011iua} for details. Moreover,
\begin{align*}
  \n{r(q)}_{w^{\ep},\sg,\infty} 
  &\le \sup_{\abs{n} < M_{m,\sg}} \e^{\ep 2n}\lin{2n}^{\sg}\abs{q_{2n}} + 
  \sup_{\abs{n} \ge M_{m,\sg}} w_{2n}\w{2n}^{\sg}\abs{r_{n}}\\
  &\le \e^{2\ep M_{m,\sg}} \n{q}_{\sg,\infty} + \n{\gm(q_{0})}_{w,\sg,\infty}. 
\end{align*}
Choosing $\ep > 0$ sufficiently small, we conclude from~\eqref{m-choice} that
\[
  \n{r(q)}_{w^{\ep},\sg,\infty} \le 2\n{q}_{\sg,\infty} + \n{\gm(q_{0})}_{s,\infty} \le m/2.
\]
Thus Proposition~\ref{abstract-diffeo} applies yielding $q\in B_{m}^{w^{\ep},\sg,\infty}$.
By the definition of $w^{\ep}$, $w_{n} \neq w_{n}^{\ep}$ holds for at most finitely many $n$, hence
\[
  \n{q}_{w^{\ep},\sg,\infty} \le C_{\ep}\n{q}_{w,\sg,\infty},
\]
where the constant $C_{\ep}\ge 1$ depends only on $\ep$ and $M_{m,\sg}$, but is independent of $q$.
Since $\n{q}_{w,\sg,\infty} = \n{q}_{s,\infty}$, it thus follows that $\n{q}_{s,\infty} \le C_{\ep}m$ for all $q\in \Iso(q_{0})$.~\qed
\end{proof}

\begin{proof}[Proof of Theorem~\ref{inv}.]
Suppose $q$ is a real valued potential in $H_{0}^{-1}$ with gap lengths $\gm(q) \in \ell^{s,\infty}$ for some $-1/2 < s \le 0$. By the preceding lemma, $\Iso(q)$ is bounded in $\FL_{0}^{s,\infty}$. Moreover, by~\cite{Kappeler:CNzeErmy}, $\Iso(q)$ is compact in $\FL_{0}^{\sg,p}$ for any $2\le p <\infty$ and $-1/2\le \sg \le 0$ with $(s-\sg)p > 1$. Consequently, $\Iso(q)$ is weak* compact in $\FL_{0}^{s,\infty}$ by Lemma~\ref{weak-star}.~\qed
\end{proof}


\section{Birkhoff coordinates on $\FL_{0}^{s,\infty}$}
\label{section5}

The aim of this section is to prove Theorem~\ref{thm:bhf}. First let us recall the results on Birkhoff coordinates on $H_{0}^{-1}$ obtained in~\cite{Kappeler:2005fb}.

\begin{thm}[\cite{Kappeler:2005fb,Kappeler:2008fl}] 
\label{thm:bhf-H-1}
There exists a complex neighborhood $\Ws$ of $H_{0}^{-1}$ within $H_{0,\C}^{-1}$ and an analytic map $\Phi\colon \Ws\to h_{0,\C}^{-1/2}$, $q\mapsto (z_{n}(q))_{n\in\Z}$ with the following properties:
\begin{enumerate}[label=(\roman{*})]
  \item $\Phi$ is canonical in the sense that $\pbr{z_{n},z_{-n}} = \int_{0}^{1}\partial_{u}z_{n}\partial_{x}\partial_{u}z_{-n}\,\dx = \ii$ for all $n\ge 1$, whereas all other brackets between coordinate functions vanish.

  \item For any $s\ge-1$, the restriction $\Phi|_{H^s_0}$ is a map 
  $\Phi|_{H^s_0} \colon H^s_0\to h^{s+1/2}_0$ which is a bianalytic diffeomorphism.

  \item The \KdV Hamiltonian $\Hm\circ\Phi^{-1}$, expressed in the new variables, is defined on $h_{0}^{3/2}$ and depends on the action variables alone. In fact, it is a real analytic function of the actions on the positive quadrant $\ell_{+}^{3,1}(\N)$,
  \[
    \ell_{+}^{3,1}(\N) \defl \setdef{(I_{n})_{n\ge 1}}{I_{n}\ge 0\forall n\ge 1,\quad \sum_{n\ge 1} n^{3}I_{n} < \infty}.~\fish
  \]
  
\end{enumerate}
\end{thm}

We will also need the following result (cf. \cite[\S\,3]{Kappeler:2005fb}).

\begin{thm}[\cite{Kappeler:2005fb}]
\label{bm.H-1}
After  shrinking, if necessary,  the  complex neighborhood $\Wp$ of $H_0^{-1}$ in $H_{0,\C}^{-1}$  of Theorem~\ref{thm:bhf-H-1} the following holds:
\begin{enumerate}[label=(\roman{*})]
\item
Let $Z_{n} = \setdef{q\in H_{0}^{-1}}{\gm_{n}^{2}(q)\neq 0}$ for $n\ge 1$. The quotient $I_{n}/\gm_n^2$, defined on $H_0^{-1} \setminus Z_n$, extends analytically to $\Wp$ for any $n\ge 1$. Moreover, for any $\ep >0$ and any $q \in \Wp$ there exists $n_0 \ge 1$ and an open neighborhood $\Wp_{q}$ of $q$ in $\Wp$ so that
\begin{equation}
  \label{In-gmn}
  \abs*{ 8 n\pi  \frac{I_n}{\gm_n^2} - 1 } \le \ep,\qquad 
  \forall n\ge n_{0} \forall p\in \Wp_{q}.
\end{equation}

\item
The Birkhoff coordinates $(z_n)_{n \in \Z}$ are analytic as maps from $\Wp$ into $\C$ and fulfill
locally uniformly in $\Wp$ and uniformly for $n \ge 1$, the estimate
\[
  \abs{z_{\pm n}} = O \p*{ \frac{|\gm_n| + |\mu_n - \tau_n|}{\sqrt{n}} }.
\]

\item
For any $q\in \Wp$ and $n \ge 1$ one has $I_{n}(q) = 0$ if and only if $\gm_{n}(q) = 0$. In particular, $\Phi(0) = 0$.~\fish
\end{enumerate} 
\end{thm}

\subsection{Birkhoff coordinates}

In~\cite{Kappeler:CNzeErmy} based on the results of~\cite{Kappeler:2008fl}, the restrictions of the Birkhoff map
\[
  \Phi\colon H_{0}^{-1}\to h_{0}^{-1/2},\qquad q\mapsto (z_{n}(q))_{n\in\Z},\quad z_{0}(q) = 0,
\]
to the Fourier Lebesgue spaces $\FL_{0}^{s,p}$, $-1/2\le s\le 0$, $2\le p <\infty$, are studied. It turns out that the arguments developed in the papers~\cite{Kappeler:CNzeErmy,Kappeler:2009uk} can be adapted to prove Theorem~\ref{thm:bhf}. As a first step we extend the results in \cite{Kappeler:CNzeErmy} for $\FL_{0}^{s,p}$, $-1/2\le s\le 0$, $2\le p <\infty$, to the case $p=\infty$. More precisely, we prove

\begin{lem}
\label{Phi-ana}
For any $-1/2< s \le 0$
\[
  \Phi_{s,\infty} \equiv \Phi\bigg|_{\FL_{0}^{s,\infty}}\colon \FL_{0}^{s,\infty}\to \ell_{0}^{s+1/2,\infty},
  \qquad q \mapsto (z_{n}(q))_{n\in\Z},
\]
is real analytic and extends analytically to an open neighborhood $\Wp_{s,\infty}$ of $\FL_{0}^{s,\infty}$ in $\FL_{0,\C}^{s,\infty}$. Its Jacobian $\ddd_{0}\Phi_{s,\infty}$ at $q = 0$ is the weighted Fourier transform 
\[
  \ddd_{0}\Phi_{s,\infty} \colon \FL_{0}^{s,\infty}\to \ell_{0}^{s+1/2,\infty},\qquad
  f\mapsto \left(\frac{1}{\sqrt{2\pi\max(\abs{n},1)}} \spii{f,e_{2n}} \right)_{n \in \Z}
\]
with inverse given by
\[
  (\ddd_{0}\Phi_{s,\infty})^{-1} \colon \ell_{0}^{s+1/2,\infty} \to \FL_{0}^{s,\infty},\quad
  (z_{n})_{n\in\Z}\mapsto \sum_{n\in\Z} \sqrt{2\pi \abs{n}}z_{n}e_{2n}.
\]
In particular, $\Phi_{s,\infty}$ is a local diffeomorphism at $q = 0$.~\fish
\end{lem}


\begin{proof}
The coordinate functions $z_{n}(q)$ are analytic functions on the complex neighborhood $\Wp\subset H_{0,\C}^{-1}$ of $H_{0}^{-1}$ of Theorem~\ref{bm.H-1}.
Since for any $-1/2 < s \le 0$, $\FL_{0,\C}^{s,\infty} \opento H_{0,\C}^{-1}$, it follows that their restrictions to $\Ws_{s,\infty} = \Ws\cap \FL_{0,\C}^{s,\infty}$ are analytic as well. Furthermore,
\[
  z_{\pm n}(q) = O\p*{\frac{\abs{\gm_{n}(q)} + \abs{\mu_{n}(q)-\tau_{n}(q)}}{\sqrt{n}}}
\]
locally uniformly on $\Ws$ and uniformly in $n\ge 1$. By the asymptotics of the periodic and Dirichlet eigenvalues of Theorems~\ref{thm:spec-fw}, $\Phi_{s,\infty}$ maps the complex neighborhood $\Ws_{s,\infty} \defl \Ws\cap \FL_{0,\C}^{s,\infty}$ of $\FL_{0}^{s,\infty}$ into the space $\ell_{0,\C}^{s+1/2,\infty}$ and is locally bounded. Using \cite[Theorem~A.3]{Kappeler:2003up}, one sees that $\Phi_{s,\infty}$ is analytic. The formulas for $\ddd_{0}\Phi_{s,\infty}$ and its inverse follow from \cite[Theorem 9.7]{Kappeler:2003up} by continuity.~\qed
\end{proof}

In a second step, following arguments used in \cite{Kappeler:2009uk}, we prove that $\Phi_{s,\infty}$ is onto.

\begin{lem}
\label{Phi-onto}
For any $-1/2< s \le 0$, the map $\Phi_{s,\infty}\colon \FL_{0}^{s,\infty}\to \ell_{0}^{s+1/2,\infty}$ is onto.~\fish
\end{lem}

\begin{proof}
Given any $z\in \ell_{0}^{s+1/2,\infty} \subset h_{0}^{-1/2}$, there exists $q\in H_{0}^{-1}$ so that $\Phi(q) = z$. Moreover, by Theorem~\ref{bm.H-1} (i) we have for all $n$ sufficiently large
\[
  \abs*{\frac{8n\pi I_{n}}{\gm_{n}^{2}}} \ge \frac{1}{2}.
\]
Since $I_{n} = z_{n}z_{-n}$ and $z\in \ell_{0}^{s+1/2,\infty}$, this implies $\gm(q)\in \ell^{s,\infty}(\N)$.  Using Theorem~\ref{inv}, we conclude that $q\in \FL_{0}^{s,\infty}$. Since by definition $\Phi_{s,\infty}$ is the restriction of the Birkhoff map $\Phi$ to $\FL_{0}^{s,\infty}$, we conclude that
\[
  \Phi_{s,\infty}(q) = z.
\]
This completes the proof.~\qed
\end{proof}

\subsection{Isospectral sets}

Recall that for any $z\in h_{0}^{-1/2}$, the torus $\Tc_{z}\subset h_{0}^{-1/2}$ was introduced in~\eqref{torus}.

\begin{lem}
\label{iso-fl-s-infty}
Suppose $q\in\FL_{0}^{s,\infty}$ with $-1/2 < s \le 0$.
\begin{enumerate}[label=(\roman{*})]
\item $\Iso(q)$ is bounded  in $\FL_{0}^{s,\infty}$.
\item
$\Phi_{s,\infty}(\Iso(q)) = \Tc_{\Phi(q)}$.
\item
If $\Phi(q)\notin c_{0}^{s} = \setdef{z\in \ell_{0}^{s,\infty}}{\w{n}^{s}z_{n}\to 0}$, then $\Phi(\Iso(q))$ is not compact in $\FL_{0}^{s,\infty}$.~\fish
\end{enumerate}
\end{lem}

\begin{proof}
(i) follows from Theorem~\ref{inv}.
According to~\cite{Kappeler:2005fb} the identity $\Phi(\Iso(q)) = \Tc_{\Phi(q)}$ holds for any $q\in H_{0}^{-1}$ and thus implies (ii). Suppose $q\in \FL_{0}^{s,\infty}$ is such that $z=\Phi(q) \notin c_{0}^{s}$. Then there exists $\ep > 0$ and a subsequence $(\nu_{n})_{n\ge 1}\subset \N$ with $\nu_{n}\to \infty$ so that
\[
  \lin{\nu_{n}}^{s}\abs{z_{\nu_{n}}}\ge \ep,\qquad \forall n\ge 1.
\]
For every $m\in\N$ define $z^{(m)}\in \Tc_{z}$ by setting $z_{0}^{(m)} = 0$ and for any $k\ge 1$, $z_{-k} = z_{k}^{(m)}$ and
\[
  z_{k}^{(m)} = \begin{cases}
  -z_{k}, & k = \nu_{m},\\
  z_{k}, & \text{otherwise}.
  \end{cases}
\]
It follows that $\n{z^{(m_{1})}-z^{(m_{2})}}_{s,\infty} \ge 2\ep$ for all $m_{1}\neq m_{2}$, hence $\Tc_{z}$ is not compact.~\qed
\end{proof}

\subsection{Weak* topology}

In this subsection we establish various properties of $\Phi_{s,\infty}$ related to the weak* topology.

\begin{lem}
\label{Phi-bounded}
For any $-1/2 < s \le 0$, the map $\Phi_{s,\infty}\colon \FL_{0}^{s,\infty}\to \ell_{0}^{s+1/2,\infty}$ is $\n{\cdot}_{s,\infty}$-norm bounded.~\fish
\end{lem}

\begin{proof}
It suffices to consider the case of the ball $B_{m}^{s,\infty}\subset \FL_{0}^{s,\infty}$ of radius $m\ge 1$. Since $B_{m}^{s,\infty}$ embeds compactly into $H_{0}^{-1}$, \eqref{In-gmn} implies that one can choose $N\ge N_{m,s}$ such that for all $q\in B_{m}^{s,\infty}$,
\[
  \frac{8n\pi I_{n}}{\gm_{n}^{2}} \le 2,\qquad n\ge N.
\]
Since $\abs{z_{n}(q)}^{2} = I_{n}$, we conclude with Lemma~\ref{gm-est} that
\begin{align*}
  \n{T_{N}\Phi(q)}_{s+1/2,\infty}
  &=
  \sup_{\abs{n}\ge N} \lin{n}^{s+1/2}\abs{z_{n}}\\
  &\le
  \sup_{\abs{n}\ge N} \lin{n}^{s}\abs{\gm_{n}}\\
  &\le
   4\n{T_{N}q}_{s,\infty} + \frac{16c_{s}}{N^{1/2-\abs{s}}}\n{q}_{s,\infty}^{2},
  \qquad q\in B_{m}^{s,\infty}.
\end{align*}
Moreover, each of the finitely many remaining coordinate functions $z_{n}(q)$, $\abs{n} < N$, is real analytic on $H_{0}^{-1}$ and hence bounded on the compact set $B_{m}^{s,\infty}$, which proves the claim.~\qed
\end{proof}

\begin{lem}
\label{Phi-wst-sequences}
For any $-1/2 < s \le 0$, the map $\Phi_{s,\infty}\colon \FL_{0}^{s,\infty}\to \ell_{0}^{s+1/2,\infty}$ maps weak* convergent sequences to weak* convergent sequences.~\fish
\end{lem}
\begin{proof}
Given $q^{(k)}\wsto q$ in $\FL_{0}^{s,\infty}$, there exists $m\ge 1$ so that $(q^{(k)})_{k\ge 1}\subset B_{m}^{s,\infty}$. Since $q^{(k)}\to q$ in $H_{0}^{-1}$ and $\Phi\colon H_{0}^{-1}\to h_{0}^{-1/2}$ is continuous, it follows that $z^{(k)}\defl\Phi(q^{(k)})\to \Phi(q)\defr z$ in $h_{0}^{-1/2}$. In particular, $z_{n}^{(k)}\to z_{n}$ for all $n\in\Z$. By the previous lemma it follows that $(z^{(k)})_{k\ge 1}$ is bounded in $\ell_{0}^{s+1/2,\infty}$ and hence $z^{(k)}\wsto z$ in $\ell_{0}^{s+1/2,\infty}$.~\qed
\end{proof}

\begin{cor}
\label{Phi-wst-continuous}
For any $-1/2 < s \le 0$ and $m\ge 1$, the map
\[
  \Phi_{s,\infty}\colon (B_{m}^{s,\infty},\sg(\FL_{0}^{s,\infty},\FL_{0}^{-s,1}))\to (\ell_{0}^{s+1/2,\infty},\sg(\ell_{0}^{s+1/2,\infty},\ell_{0}^{-(s+1/2),1}))
\]
is a homeomorphism onto its image.~\fish
\end{cor}

\begin{proof}
By Lemma~\ref{weak-star}, $(B_{m}^{s,\infty},\sg(\FL_{0}^{s,\infty},\FL_{0}^{-s,1}))$ is metrizable. Hence by Lemma~\ref{Phi-bounded} and Lemma~\ref{Phi-wst-sequences}, the map $\Phi\colon (B_{m}^{s,\infty},\sg(\FL_{0}^{s,\infty},\FL_{0}^{-s,1}))\to (\ell_{0}^{s+1/2,\infty},\sg(\ell_{0}^{s+1/2,\infty},\ell_{0}^{-(s+1/2),1}))$ is continuous. Since $\Phi_{s,\infty}\colon \FL_{0}^{s,\infty}\to \ell_{0}^{s+1/2,\infty}$ is bijective and $B_{m}^{s,\infty}$ is compact with respect to the weak* topology, the claim follows.~\qed
\end{proof}

\subsection{Proof of Theorem~\ref{thm:bhf} and asymptotics of the KdV frequencies}
\label{ss:kdv-freq}

\begin{proof}[Proof of Theorem~\ref{thm:bhf}.]
The claim follows from Lemma~\ref{Phi-ana}, Lemma~\ref{Phi-onto},  Lemma~\ref{iso-fl-s-infty}, Lemma~\ref{Phi-bounded}, and Corollary~\ref{Phi-wst-continuous}.~\qed
\end{proof}

Recall that in \cite{Kappeler:2006fr} the KdV frequencies $\om_{n} = \partial_{I_{n}}\Hm$ have been proved to extend real analytically to $H_{0}^{-1}$ -- see also \cite{Kappeler:2016uj} for more recent results.

\begin{lem}
\label{kdv-freq-pseudo}
Uniformly on $\n{\cdot}_{s,\infty}$-norm bounded subsets of $\FL_{0}^{s,\infty}$, $-1/2 < s \le 0$,
\[
  \om_{n} = (2n\pi)^{3} - 6I_{n} + o(1).~\fish
\]
\end{lem}

\begin{proof}
The claim follows immediately from \cite[Theorem~3.6]{Kappeler:2016uj} and the fact that $\FL_{0}^{s,\infty}$ embeds compactly into $\FL^{-1/2,p}$ if $(s+1/2)p > 1$.~\qed
\end{proof}

\section{Proofs of Theorems~\ref{thm:kdv-wp} and ~\ref{thm:invariant}}
\label{s:main-proofs}

\begin{proof}[Proof of Theorem~\ref{thm:kdv-wp}.]
According to \cite{Kappeler:2006fr}, for any $q\in \FL_{0}^{s,\infty}\opento H_{0}^{-1}$, the solution curve $t\mapsto S(q)(t)\in H_{0}^{-1}$ exists globally in time and is contained in $\Iso(q)$. Since the latter is $\n{\cdot}_{s,\infty}$-norm bounded by Lemma~\ref{iso-bounded}, the solution curve is uniformly $\n{\cdot}_{s,\infty}$-norm bounded in time,
\[
  \sup_{t\in\R}\n{S(q)(t)}_{s,\infty} \le \sup_{\tilde q\in \Iso(q)}\n{\tilde q} < \infty.
\]
By \cite{Kappeler:2006fr}, any coordinate function $t\mapsto (S(q))_{n}(t)$, $n\in\Z$, is continuous and hence $\R\mapsto (\FL_{0}^{s,\infty},\tau_{w*})$, $t\mapsto S(q)(t)$ is a continuous map.~\qed
\end{proof}

\begin{proof}[Proof of Theorem~\ref{thm:invariant}.]
Suppose $V\subset\FL_{0}^{s,\infty}$ is a $\n{\cdot}_{s,\infty}$-norm bounded subset. Then there exists $m \ge 1$ so that $V\subset B_{m}^{s,\infty}$ and the weak* topology induced on $B_{m}^{s,\infty}$ coincides with the norm topology induced from $\FL_{0}^{\sg,p}$ provided $(s-\sg)p > 1$ -- see Lemma~\ref{weak-star}. Since by \cite{Kappeler:CNzeErmy}, for any $-1/2\le \sg \le 0$, $2\le p < \infty$, the map
\[
  \Sc\colon (V,\n{\cdot}_{\sg,p})\to C([-T,T],(V,\n{\cdot}_{\sg,p}))
\]
is continuous, it follows that
\[
  \Sc\colon (V,\tau_{w*})\to C([-T,T],(V,\tau_{w*}))
\]
is continuous as well.~\qed
\end{proof}

\begin{proof}[Proof of Remark~\ref{rem-airy}.]
Since by Lemma~\ref{Phi-ana}, the Birkhoff map $\Phi$ is a local diffeomorphism near $0$, it suffices to show for generic small initial data $q$ in $\FL_{0}^{s,\infty}$ that the solution curve $t\mapsto \Sc(q)(t)$, expressed in Birkhoff coordinates, is not continuous. But this latter claim follows in a straightforward way from the asymptotics of the KdV frequencies of Lemma~\ref{kdv-freq-pseudo}.~\qed
\end{proof}

\section{Wiener Algebra}
\label{s:wiener}

It turns out that by our methods we can also prove that the KdV equation is globally in time $C^{\om}$-wellposed on $\FL_{0}^{0,1}$, referred to as Wiener algebra. Actually, we prove such a result for any Fourier Lebesgue space $\FL_{0}^{N,1}$ with $N\in\Z_{\ge 0}$.

\subsection{Birkhoff coordinates}

In a first step we prove that $\FL_{0}^{N,1}$ admits global Birkhoff coordinates. More precisely, we show

\begin{thm}
\label{thm:bhf-wiener} 
For any $N\in\Z_{\ge 0}$, the restriction $\Phi_{N,1}$ of the Birkhoff map $\Phi$ to $\FL_{0}^{N,1}$ takes values in $\ell_{0}^{N + 1/2,1}$ and $\Phi_{N,1}\colon \FL_{0}^{N,1} \to \ell_{0}^{N+1/2,1}$ is a real analytic diffeomorphism, and therefore provides global Birkhoff coordinates on $\FL_{0}^{N,1}$.~\fish
\end{thm}

Before we prove Theorem~\ref{thm:bhf-wiener}, we need to review the spectral theory of the Schrödinger operator $-\partial_{x}^{2}+q$ for $q\in \FL_{0}^{N,1}$.

\subsection{Spectral Theory}

The spectral theory of the operator $L(q) = -\partial_{x}^{2} + q$ for $q\in\FL_{0}^{N,1}$ was considered in \cite{Poschel:2011iua} where the following results were shown.

\begin{thm}[{\cite[Theorem~1 \& 4]{Poschel:2011iua}}]
\label{thm:spec-fw-wiener}
Let $N\in\Z_{\ge0}$.
\begin{enumerate}[label=(\roman{*})]
\item
For any $q\in\FL_{0,\C}^{N,1}$, the sequence of gap lengths $(\gm_{n}(q))_{n\ge 1}$, defined in~\eqref{gm-tau} is in $\ell_{\C}^{N,1}(\N)$ and the map
\[
  \FL_{0,\C}^{N,1}\to \ell_{\C}^{N,1}(\N),\qquad q \mapsto (\gm_{n}(q))_{n\ge 1},
\]
is uniformly bounded on bounded subsets.

\item
For any $q\in\FL_{0,\C}^{N,1}$, the sequence $(\tau_{n}-\mu_{n}(q))_{n\ge 1}$ -- cf. \eqref{gm-tau}-\eqref{mu-asymptotics-H-1} -- is in $\ell_{\C}^{N,1}(\N)$ and the map
\[
  \FL_{0,\C}^{N,1}\to \ell_{\C}^{N,1}(\N),\qquad q \mapsto (\tau_{n}(q)-\mu_{n}(q))_{n\ge 1},
\]
is uniformly bounded on bounded subsets.~\fish
\end{enumerate}
\end{thm}

In addition, the following spectral characterization for a potential $q\in L^{2}$ to be in $\FL_{0}^{N,1}$ holds.

\begin{thm}[{\cite[Theorem~3]{Poschel:2011iua}}]
\label{thm:spec-inv-wiener}
Let $q\in L_{0}^{2}$ and assume that $(\gm_{n}(q))_{n\ge 1}\in \ell_{\R}^{N,1}$ for some $N\in\Z_{\ge 0}$. Then $q \in \FL_{0}^{N,1}$ and $\Iso(q) \subset \FL_{0}^{N,1}$.~\fish
\end{thm}

\subsection{Proof of Theorem~\ref{thm:bhf-wiener}}

Theorem~\ref{thm:bhf-wiener} will follow from the following lemmas.

\begin{lem}
\label{Phi-ana-wiener}
For any $N\in\Z_{\ge 0}$
\[
  \Phi_{N,1} \equiv \Phi\bigg|_{\FL_{0}^{N,1}}\colon \FL_{0}^{N,1}\to \ell_{0}^{N+1/2,1},
  \qquad q \mapsto (z_{n}(q))_{n\in\Z},
\]
is real analytic and extends analytically to an open neighborhood $\Wp_{N,1}$ of $\FL_{0}^{N,1}$ in $\FL_{0,\C}^{N,1}$.~\fish
\end{lem}


\begin{proof}
The coordinate functions $z_{n}(q) = (\Phi(q))_{n}$, $n\in\Z$, are analytic functions on the complex neighborhood $\Wp\subset H_{0,\C}^{-1}$ of $H_{0}^{-1}$ of Theorem~\ref{bm.H-1}. Furthermore,
\[
  z_{\pm n}(q) = O\p*{\frac{\abs{\gm_{n}(q)} + \abs{\mu_{n}(q)-\tau_{n}(q)}}{\sqrt{n}}}
\]
locally uniformly on $\Ws$ and uniformly in $n\ge 1$. By the asymptotics of the periodic and Dirichlet eigenvalues of Theorem~\ref{thm:spec-fw-wiener}, $\Phi_{N,1}$ maps the complex neighborhood $\Ws_{N,1} \defl \Ws\cap \FL_{0,\C}^{N,1}$ of $\FL_{0}^{N,1}$ into the space $\ell_{0,\C}^{N+1/2,1}$ and is locally bounded. It then follows from~\cite[Theorem~A.4]{Grebert:2014iq} that for any $\xi\in \ell_{0,\C}^{-(N+1/2),\infty}$, the map $q\mapsto \lin{\xi,\Phi(q)}$ is analytic on $\Ws\cap \ell_{0,\C}^{N,1}$ implying that $\Phi\colon \Ws_{N,1}\to \ell_{0,\C}^{N+1/2,1}$ is weakly analytic. Hence by~\cite[Theorem~A.3]{Grebert:2014iq}, $\Phi_{N,1}$ is analytic.~\qed
\end{proof}

Next, following arguments used in \cite{Kappeler:2009uk}, we prove that $\Phi_{N,1}$ is onto.

\begin{lem}
\label{Phi-onto-wiener}
For any $N\in\Z_{\ge 0}$, the map $\Phi_{N,1}\colon \FL_{0}^{N,1}\to \ell_{0}^{N+1/2,1}$ is onto.~\fish
\end{lem}

\begin{proof}
For any $z\in \ell_{0}^{N+1/2,1} \subset h_{0}^{1/2}$, there exists $q\in L_{0}^{2}$ so that $\Phi(q) = z$. Moreover, by Theorem~\ref{bm.H-1} (i) we have for all $n$ sufficiently large
\[
  \abs*{\frac{8n\pi I_{n}}{\gm_{n}^{2}}} \ge \frac{1}{2}.
\]
Since $I_{n} = z_{n}z_{-n}$ and $z\in \ell_{0}^{N+1/2,1}$, this implies $\gm(q)\in \ell^{N,1}(\N)$. Using Theorem~\ref{thm:spec-inv-wiener}, we conclude that $q\in \FL_{0}^{N,1}$. Since by definition $\Phi_{N,1}$ is the restriction of the Birkhoff map $\Phi$ to $\FL_{0}^{N,1}$, we conclude
\[
  \Phi_{N,1}(q) = z.
\]
This completes the proof.~\qed
\end{proof}

\begin{lem}
\label{dPhi-iso-wiener}
For any $q\in \FL_{0}^{N,1}$ with $N\in\Z_{\ge 0}$,
\[
  \ddd_{q}\Phi_{N,1}\colon \FL_{0}^{N,1}\to \ell_{0}^{N+1/2,1}
\]
is a linear isomorphism.~\fish
\end{lem}

\begin{proof}
By Theorem~\ref{thm:bhf-H-1}, $\ddd_{q}\Phi\colon H_{0}^{-1}\to h_{0}^{-1/2}$ is a linear isomorphism for any $q\in H_{0}^{-1}$. 
Since $\ddd_{q}\Phi_{N,1} = \ddd_{q}\Phi\big|_{\FL_{0}^{N,1}}$ for any $q\in \FL_{0}^{N,1}$, it follows from Lemma~\ref{Phi-ana-wiener} that $\ddd_{q}\Phi_{N,1}\colon \FL_{0}^{N,1}\to \ell_{0}^{N+1/2,1}$ is one-to one.
To show that $\ddd_{q}\Phi_{N,1}$ is onto, note that by Theorem~\ref{thm:bhf-H-1}, $\ddd_{0}\Phi_{N,1}\colon \FL_{0}^{N,1}\to \ell_{0}^{N+1/2,1}$ is a weighted Fourier transform and hence a linear isomorphism. It therefore suffices to show that $\ddd_{q}\Phi_{N,1}-\ddd_{0}\Phi_{N,1}\colon \FL_{0}^{N,1}\to \ell_{0}^{N+1/2,1}$ is a compact operator implying that $\ddd_{q}\Phi_{N,1}$ is a Fredholm operator of index zero and thus a linear isomorphism.
To show that $\ddd_{q}\Phi_{N,1}-\ddd_{0}\Phi_{N,1}\colon \FL_{0}^{N,1}\to \ell_{0}^{N+1/2,1}$ is compact we use that by \cite[Theorem~1.4]{Kappeler:2013bt}, for any $q\in H_{0}^{N}$, the restriction of $\ddd_{q}\Phi$ to $H_{0}^{N}$ has the property that
$
  \ddd_{q}\Phi-\ddd_{0}\Phi\colon H_{0}^{N}\to h_{0}^{N+3/2}
$
is a bounded linear operator.
In view of the fact that $\FL_{0}^{N,1}\opento H_{0}^{N}$ is bounded and $h_{0}^{N+3/2} \opento_{c} \ell_{0}^{N+1/2,1}$ is compact, it follows that
\[
  \ddd_{q}\Phi_{N,1}-\ddd_{0}\Phi_{N,1}\colon \FL_{0}^{N,1}\to \ell_{0}^{N+1/2,1}
\]
is a compact operator.~\qed
\end{proof}

\subsection{Frequencies}

Finally, we need to consider the KdV frequencies introduced in Subsection~\ref{ss:kdv-freq}. They are viewed either as functions on $\FL_{0}^{N,1}$ or as functions of the Birkhoff coordinates on $\ell_{0}^{N+1/2,1}$.

\begin{lem}
\label{frequencies-wiener}
The KdV frequencies $\om_{n}$, $n\ge1$, admit a real analytic extension to a common complex neighborhood $\Ws^{N,1}$ of $\FL_{0}^{N,1}$, $N\in\Z_{\ge 0}$, and for any $r > 1$ have the asymptotic behavior
\[
  \om_{n} - 8n^{3}\pi^{3} = \ell_{n}^{r},
\]
locally uniformly on $\Ws^{N,1}$.~\fish
\end{lem}

\begin{proof}
Since $\FL_{0,\C}^{N,1}\opento L_{0,\C}^{2}$ this is an immediate consequence of \cite[Theorem~3.6]{Kappeler:2016uj}.~\qed
\end{proof}

\subsection{Wellposedness}

We are now in position to prove that the KdV equation is globally in time $C^{\om}$-wellposed on $\FL_{0}^{N,1}$ for any $N\in\Z_{\ge 0}$. First we consider the KdV equation in Birkhoff coordinates.
Let $\Sc_{\Ph}\colon (t,z) \mapsto (\ph_{n}^{t}(z))_{n\in\Z}$ denote the flow in Birkhoff coordinates  with coordinate functions
\[
  \ph_{n}^{t}(z) = \e^{\ii \om_{n}(z)t}z_{n},\qquad n\in\Z.
\]

\begin{lem}
\label{kdv-wp-wiener-bhf}
For any $N\in\Z_{\ge 0}$ and $T > 0$, the map
\[
  \Sc_{\Phi}\colon \ell_{0}^{N+1/2,1}\to C([-T,T],\ell_{0}^{N+1/2,1}),
  \qquad z\mapsto (t\mapsto \Sc_{\Ph}(t,z)),
\]
is real-analytic.~\fish
\end{lem}

\begin{proof}
Since $\om_{n} - 8n^{3}\pi^{3} = o(1)$ locally uniformly, this is an immediate consequence of \cite[Theorem~E.1]{Kappeler:2016uj}.~\qed
\end{proof}

\begin{thm}
\label{thm:kdv-wp-wiener}
For any $N\in\Z_{\ge 0}$, the KdV equation is globally in time $C^{\om}$-wellposed on $\FL_{0}^{N,1}$. More precisely, for any $T > 0$, the map
\[
  \Sc\colon \FL_{0}^{N,1}\to C([-T,T],\FL_{0}^{N,1}),\qquad q\mapsto (t\mapsto\Sc(t,q)),
\]
is real analytic.~\fish
\end{thm}

\begin{proof}
The claim follows immediately from the Lemma~\ref{kdv-wp-wiener-bhf} and the fact established in Theorem~\ref{thm:bhf-wiener} that the Birkhoff map is a real analytic diffeomorphism $\Ph\colon \FL_{0}^{N,1}\to \ell_{0}^{N+1/2,1}$.~\qed
\end{proof}

\appendix

\section{Auxiliaries}
\label{app:aux}

\begin{lem}
\label{hilbert-sum}
For any $1/2< \sg < \infty$ there exists a constant $C_{\sg} > 0$ so that for any $n\ge 1$,
$\sum_{\abs{m}\neq n} \frac{1}{\abs{m^{2}-n^{2}}^{\sg}}$ is bounded by $C_{\sg}/n^{2\sg-1}$ if $1/2 < \sg < 1$, $C_{\sg} \frac{\log \w{n}}{n}$ if $\sg = 1$, and $C_{\sg}/n^{\sg}$ if $\sg > 1$.~\fish
\end{lem}

\begin{proof}
\cite[Lemma~A.1]{Kappeler:CNzeErmy}.\qed
\end{proof}

For any $s\in\R$ and $1\le p \le \infty$ denote by $\ell_{\C}^{s,p} \equiv \ell^{s,p}(\Z,\C)$ the sequence space
\[
  \ell_{\C}^{s,p} = \setdef{z = (z_{k})_{k\in\Z}\subset \C}{\n{z}_{s,p} < \infty}.
\]

\begin{lem}
\label{ell-embedding}
Suppose $-1/2 < s \le 0$. For any $-1 \le \sg < s$ and $2\le p < \infty$ with $(s-\sg)p > 1$  one has $\ell_{\C}^{s,\infty} \opento \ell_{\C}^{\sg,p}$ and the embedding is compact. In particular, for any $\ep > 0$, $\ell_{\C}^{s,\infty}\opento h_{\C}^{-1/2+s-\ep}$.~\fish
\end{lem}

\begin{proof}
By Hölder's inequality
\[
  \p*{\sum_{m\in\Z} \lin{m}^{\sg p}\abs{a_{m}}^{p}}^{1/p}
  \le
  \p*{\sup_{m\in\Z} \lin{m}^{s}\abs{a_{m}}}
  \p*{\sum_{m\in\Z} \lin{m}^{-(s-\sg)p}}^{1/2},
\]
provided $(s-\sg)p > 1$.
Hence $\ell_{\C}^{s,\infty}\opento \ell_{\C}^{\sg,p}$. The compactness follows from the well known characterization of compact subsets in $\ell^{p}$.~\qed
\end{proof}

The following result is well known -- cf. \cite[Lemma~20]{Kappeler:CNzeErmy}.

\begin{lem}
\label{convolution}
\begin{enumerate}[label=(\roman{*})]
\item Let $-1 \le t < -1/2$. For $a = (a_{m})_{m\in \Z}\in h_{\C}^{t}$ and $b=(b_{m})_{m\in\Z}\in h_{\C}^{1}$, the convolution $a*b = (\sum_{m\in\Z} a_{n-m}b_{m})_{n\in\Z}$ is well defined and
\[
  \n{a*b}_{t,2} \le C_{t}\n{a}_{t,2}\n{b}_{1,2}.
\]
\item Let $-1/2 \le s \le 0$ and $-s-3/2 < t < 0$. For any $a = (a_{m})_{m\in \Z}\in \ell_{\C}^{s,\infty}$ and $b=(b_{m})_{m\in\Z}\in h_{\C}^{t+2}$,
\[
  \n{a*b}_{s,\infty} \le C_{s,t}\n{a}_{s,\infty}\n{b}_{t+2,2}.~\fish
\]
\end{enumerate}
\end{lem}

The following result is a version of the inverse function theorem.

\begin{lem}
\label{ift}
Let $E$ be a complex Banach space and denote for $r > 0$, $B_{r} = \setdef{x\in E}{\n{x} \le r}$.
If $f\colon B_{m}\to E$ is analytic for some $m\ge 1$, and
\[
  \sup_{x\in B_{m}}\abs{f(x)-x} \le m/8,
\]
then $f$ is an analytic diffeomorphism onto its image, and this image covers $B_{m/2}$.~\fish
\end{lem}

\section{Facts on the weak * topology}
\label{app:weak-star}

In this subsection we collect various properties of the weak* topology $\tau_{w*}=\sg(\ell_{0}^{s,\infty},\ell_{0}^{-s,1})$ on $\ell_{0}^{s,\infty}$ needed in the course of the paper.

\begin{lem}
\label{weak-star}
Let $s\in\R$.
\begin{enumerate}[label=(\roman{*})]
\item The closed unit ball of $\ell_{0}^{s,\infty}$ is weak* compact, weak* sequentially compact, and the topology induced by the weak* topology on this ball is metrizable.
\item 
For any sequence $(x^{(m)})_{m\ge 1}\subset\ell_{0}^{s,\infty}$ and $x\in\ell_{0}^{s,\infty}$ the following statements are equivalent:
\begin{enumerate}
\item $(x^{(m)})$ is weak* convergent to $x$.
\item $(x^{(m)})$ is $\n{\cdot}_{s,\infty}$-norm bounded and componentwise convergent, i.e.
\[
  \sup_{m\ge 1}\n{x^{(m)}}_{s,\infty} < \infty,\qquad \lim\limits_{m\to\infty}x_{n}^{(m)} = x_{n},\quad \forall n\in\Z.
\]
\item $(x^{(m)})$ is $\n{\cdot}_{s,\infty}$-norm bounded and $x^{(m)}\to x$ in $\ell_{0}^{\sg,p}$ for some $\sg < s$ and $1\le p < \infty$ with $(s-\sg)p > 1$.
\end{enumerate}

\item For any subset $A\subset \ell_{0}^{s,\infty}$ the following statements are equivalent:
\begin{enumerate}
\item $A$ is weak* compact.
\item $A$ is weak* sequentially compact.
\item $A$ is $\n{\cdot}_{s,\infty}$-norm bounded and weak* closed.
\item $A$ is $\n{\cdot}_{s,\infty}$-norm bounded and $A$ is a compact subset of $\ell_{0}^{\sg,p}$ for some $\sg < s$ and $1\le p < \infty$ with $(s-\sg)p > 1$.
\end{enumerate}
\item On any $\n{\cdot}_{s,\infty}$-norm bounded subset $A\subset\ell_{0}^{s,\infty}$, the topology induced by the $\n{\cdot}_{\sg,p}$-norm, provided $(s-\sg)p > 1$, coincides with the topology induced by the weak* topology of $\ell_{0}^{s,\infty}$.~\fish
\end{enumerate}
\end{lem}

\section{Schrödinger Operators}
\label{a:hill-op}

In this appendix we review definitions and properties of Schrödinger operators $-\partial_{x}^{2}+q$ with a singular potential $q$ used in Section~\ref{s:spectral-theory} -- see e.g. \cite{Kappeler:2001bi} and \cite{Savchuk:2005wb}.

\emph{Boundary conditions.} Denote by $H^{1}_{\C}[0,1] = H^{1}([0,1],\C)$ the Sobolev space of functions $f\colon [0,1]\to \C$ which together with their distributional derivative $\partial_{x}f$ are in $L^{2}_{\C}[0,1]$. On $H^{1}_{\C}[0,1]$ we define the following three boundary conditions (bc),
\[
  (\per+)\quad f(1) = f(0);\quad
  (\per-)\quad  f(1) = -f(0);\quad
  (\dir)\quad   f(1) = f(0) = 0.
\]
The corresponding subspaces of $H^{1}_{\C}[0,1]$ are defined by
\[
  H_{bc}^{1} = \setdef{f\in H^{1}_{\C}[0,1]}{f\text{ satisfies (bc)}},
\]
and their duals are denoted by $H_{bc}^{-1} \defl (H_{bc}^{1})'$. Note that $H_{\per+}^{1}$ can be canonically identified with the Sobolev space $H^{1}(\R/\Z,\C)$ of 1-periodic functions $f\colon \R\to \C$ which together with their distributional derivative are in $L^{2}_{\loc}(\R,\C)$. Analogously, $H_{\per-}^{1}$ can be identified with the subspace of $H^{1}(\R/2\Z,\C)$ consisting of functions $f\colon \R\to \C$ with $f,\partial_{x}f\in L_{\loc}^{2}(\R,\C)$ satisfying $f(x+1) = -f(x)$ for all $x\in \R$. In the sequel we will not distinguish these pairs of spaces.
Furthermore, note that $H_{\dir}^{1}$ is a subspace of $H^{1}_{\per+}$ as well as of $H^{1}_{\per-}$. Denote by $\lin{\cdot,\cdot}_{bc}$ the extension of the $L^{2}$-inner product $\spi{f,g} = \int_{0}^{1} f(x)\ob{g(x)}\,\dx$ to a sesquilinear pairing of $H_{bc}^{-1}$ and $H_{bc}^{1}$. Finally, we record that the multiplication
\begin{equation}
  \label{H1-mul}
  H^{1}_{bc}\times H^{1}_{bc}\to H^{1}_{\per+},\qquad (f,g) \mapsto fg,
\end{equation}
and the complex conjugation $H_{bc}^{1}\to H_{bc}^{1}$, $f\mapsto \ob{f}$ are bounded operators.

\emph{Multiplication operators.}
For $q\in H_{\per+}^{-1}$ define the operator $V_{bc}$ of multiplication by $q$, $V_{bc}\colon H^{1}_{bc}\to H^{-1}_{bc}$ as follows: for any $f\in H_{bc}^{1}$, $V_{bc}f$ is the element in $H_{bc}^{-1}$ given by
\[
  \lin{V_{bc}f,g}_{bc} \defl \lin{q,\ob{f}g}_{\per+},\qquad g\in H_{bc}^{1}.
\]
In view of~\eqref{H1-mul}, $V_{bc}$ is a well defined bounded linear operator.

\begin{lem}
\label{V-dir-V-per}
  Let $q\in H^{-1}_{\per+}$. For any $g\in H_{\dir}^{1}$, the restriction 
  $(V_{\per\pm }g)\big|_{H_{\dir}^{1}}\colon H_{\dir}^{1}\to \C$ coincides with $V_{\dir}g\colon H_{\dir}^{1}\to \C$.~\fish
\end{lem}
\begin{proof}
Since any $h\in H_{\dir}^{1}$ is also in $H^{1}_{\per+}$, the definitions of $V_{\per+}$ and $V_{\dir}$ imply
\[
  \lin{V_{\per+}g,h}_{\per+} = \lin{q,\ob{g}h}_{\per+} = \lin{V_{\dir}g,h}_{\dir},
\]
which gives $(V_{\per+}g)\big|_{H_{\dir}^{1}} = V_{\dir}g$. Similarly,  one sees that $V_{\per-}g\big|_{H_{\dir}^{1}} = V_{\dir}g$.~\qed
\end{proof}

It is convenient to introduce also the space $H_{\per+}^{1}\oplus H_{\per-}^{1}$ and define the multiplication operator $V$ of multiplication by $q$
\[
  V\colon H_{\per+}^{1}\oplus H_{\per-}^{1}\to H_{\per+}^{-1}\oplus H_{\per-}^{-1},\qquad
  (f,g) \mapsto (V_{\per+}f,V_{\per-}g).
\]
We note that $H_{\per+}^{1}\oplus H_{\per-}^{1}$ can be canonically identified with $H^{1}(\R/2\Z,\C)$,
\[
  H^{1}(\R/2\Z,\C) \to H_{\per+}^{1}\oplus H_{\per-}^{1},\qquad f\mapsto (f^{+},f^{-}),
\]
where $f^{+}(x) = \frac{1}{2}(f(x) + f(x+1))$ and $f^{-}(x) = \frac{1}{2}(f(x) - f(x+1))$.
Its dual is denoted by $H^{-1}(\R/2\Z,\C)$.

\emph{Fourier basis.} The spaces $H_{\per\pm}^{1}$, $H^{1}(\R/2\Z,\C)$ and $H_{\dir}^{1}$ and their duals admit the following standard Fourier basis.
Recall from Appendix~\ref{app:aux} that for any $s\in\R$ and $1\le p \le \infty$, we denote by $\ell_{\C}^{s,p} \equiv \ell^{s,p}(\Z,\C)$ the sequence space
\[
  \ell_{\C}^{s,p} = \setdef{z = (z_{k})_{k\in\Z}\subset \C}{\n{z}_{s,p} < \infty}.
\]

\underline{Basis for $H_{\per+}^{1}$, $H_{\per+}^{-1}$.} Any element $f\in H^{1}_{\per+}$ [$H^{-1}_{\per+}$] can be represented as $f = \sum_{m\in\Z} f_{m}e_{m}$, $e_{m}(x)\defl\e^{\ii m\pi x}$, where $(f_{m})_{m\in\Z}\in h_{\C}^{1}$ [$h_{\C}^{-1}$] and
\[
  f_{2m} = \lin{f,e_{2m}}_{\per+},\qquad f_{2m+1} = 0,\qquad \forall m\in \Z.
\]
Furthermore, for any $q = \sum_{m\in\Z} q_{m}e_{m}\in H^{-1}_{\per+}$,
\[
  V_{\per+}f = \sum_{n\in\Z}\p*{ \sum_{m\in \Z} q_{n-m}f_{m} }e_{n} \in H_{\per+}^{-1}.
\]
Note that by Lemma~\ref{convolution}, $(\sum_{m\in \Z} q_{n-m}f_{m} )_{n\in\Z}$ is in $h_{\C}^{-1}$.

\underline{Basis for $H_{\per-}^{1}$, $H_{\per-}^{-1}$.} Any element $f\in H^{1}_{\per-}$ [$H^{-1}_{\per-}$] can be represented as $f = \sum_{m\in\Z} f_{m}e_{m}$ where $(f_{m})_{m\in\Z}\in h_{\C}^{1}$ [$h_{\C}^{-1}$] and
\[
  f_{2m+1} = \lin{f,e_{2m+1}}_{\per-},\qquad f_{2m} = 0,\qquad \forall m\in \Z.
\]
Similarly, for any $q = \sum_{m\in\Z} q_{m}e_{m}\in H^{-1}_{\per+}$,
\[
  V_{\per-}f = \sum_{n\in\Z}\p*{ \sum_{m\in \Z} q_{n-m}f_{m} }e_{n} \in H_{\per-}^{-1}.
\]

\underline{Basis for $H^{1}(\R/2\Z,\C)$, $H^{-1}(\R/2\Z,\C)$.} Any element $f\in H^{1}(\R/2\Z,\C)$ [$H^{-1}(\R/2\Z,\C)$] can be represented as $f = \sum_{m\in\Z} f_{m}e_{m}$ where $f_{m} = \spii{f,e_{m}}$. Here $\spii{f,g} \defl \frac{1}{2}\int_{0}^{2}f(x)\ob{g(x)}\,\dx$ denotes the normalized $L^{2}$-inner product on $[0,2]$ extended to a sesquilinear pairing between $H^{1}(\R/2Z,\R)$ and its dual. In particular, for $f\in H^{1}(\R/2\Z,\C)$, and $m\in\Z$,
\[
  \spii{f,e_{m}} = \frac{1}{2}\int_{0}^{2} f(x)\e^{-\ii m\pi x}\,\dx.
\]
For any $q = \sum_{m\in\Z} q_{m}e_{m}\in H_{\per+}^{-1}\opento H^{1}(\R/2\Z,\C)$
\[
  Vf = \sum_{n\in\Z} \p*{ \sum_{m\in\Z} q_{n-m} f_{m} }e_{n}\in H^{-1}(\R/2\Z,\C).
\]

\underline{Basis for $H_{\dir}^{1}$, $H_{\dir}^{-1}$:} Note that $(\sqrt{2}\sin(m\pi x))_{m\ge 1}$ is an $L^{2}$-orthonormal basis of $L_{0}^{2}([0,1],\C)$. Hence any element $f\in H_{\dir}^{1}$ can be represented as
\[
  f(x) = \sum_{m\ge 1} \spi{f,s_{m}}s_{m}(x) = 
  \frac{1}{2}\sum_{m\in \Z} f_{m}^{\sin}s_{m}(x),\qquad s_{m}(x) = \sqrt{2}\sin(m\pi x),
\]
where $f_{m}^{\sin} = \int_{0}^{1} f(x)s_{m}(x)\,\dx$, $m\in\Z$. For any element $g\in H_{\dir}^{-1}$ one gets by duality
\[
  g = \frac{1}{2}\sum_{m\in\Z} g_{m}^{\sin} s_{m},\qquad g_{m}^{\sin} = \lin{g,s_{m}}_{\dir}.
\]
One verifies that $g_{-m}^{\sin} = -g_{m}^{\sin}$ for all $m\in\Z$ and $\sum_{m\in \Z} \w{m}^{-2}\abs{g_{m}^{\sin}}^{2} < \infty$.
For any $q\in H_{0,\C}^{-1}$ with $\n{q}_{t,2} < \infty$ and $-1 < t < -1/2$, we need to expand for a given $f\in H_{\dir}^{1}$, $V_{\dir}f\in H_{\dir}^{-1}$ in its sine series $\frac{1}{2}\sum_{m\in \Z} (V_{\dir}f)_{m}^{\sin} s_{m}$ where by the definition of $V_{\dir}$
\[
  (V_{\dir}f)_{m}^{\sin} = \lin{V_{\dir}f,s_{m}}_{\dir} = 
  \lin{q,\ob{f}s_{m}}_{\per+} = \frac{1}{2}\sum_{n\in\Z} f_{n}^{\sin}\lin{q,s_{n}s_{m}}_{\per+}.
\]
Using that $f_{-n}^{\sin} = -f_{n}^{\sin}$ for any $n\in\Z$ and
\[
  s_{m}(x)s_{n}(x) = \cos((m-n)\pi x) - \cos((m+n)\pi x)
\]
it follows that for any $m\in\Z$
\[
  \frac{1}{2}\sum_{\atop{m-n\text{ even}}{n\in\Z}} f_{n}^{\sin} \lin{q,s_{n}s_{m}}_{\per+} = 
  \sum_{\atop{m-n\text{ even}}{n\in\Z}} f_{n}^{\sin} \lin{q,\cos((m-n)\pi x)}_{\per+}.
\]
Note that $\lin{q,\cos((m-n)\pi x)}_{\per+}$ is well defined as $\cos((m-n)\pi x)\in H_{\per+}^{1}$ if $m-n$ is even. If $m-n$ is odd, we decompose the difference of the cosines in $H_{\per+}^{1}$ as follows
\begin{align*}
  &\cos((m-n)\pi x) - \cos((m+n)\pi x) \\
  &\qquad = 
  \p*{\cos((m-n)\pi x) - \cos(\pi x)} - \p*{\cos((m+n)\pi x) - \cos(\pi x)}
\end{align*}
and then obtain, using again that $f_{-n}^{\sin} = -f_{n}^{\sin}$ for all $n\in\Z$,
\[
  \frac{1}{2}\sum_{\atop{m-n\text{ odd}}{n\in\Z}} f_{n}^{\sin} \lin{q,s_{n}s_{m}}_{\per+} = 
  \sum_{\atop{m-n\text{ odd}}{n\in\Z}} f_{n}^{\sin} \lin{q,\cos((m-n)\pi x) - \cos(\pi x)}_{\per+}.
\]
Altogether we have shown that
\[
  V_{\dir}f = \frac{1}{2}\sum_{m\in \Z} \p*{\sum_{n\in\Z} q_{m-n}^{\cos}f_{n}^{\sin}} s_{m},
\]
where
\begin{equation}
  \label{q-cos-coeff}
  q_{k}^{\cos} = 
  \begin{cases}
  \lin{q,\cos(k\pi x)}_{\per+}, & \text{if }k\in\Z\text{ even},\\
  \lin{q,\cos(k\pi x)-\cos(\pi x)}_{\per+}, & \text{if }k\in\Z\text{ odd}.
  \end{cases}
\end{equation}
Since by assumption $\n{q}_{t,2} < \infty$ with $-1 < t < -1/2$, one argues as in \cite[Proposition 3.4]{Kappeler:2001bi}, using duality and interpolation, that
\begin{equation}
  \label{H-dir-per-est}
  \p*{ \sum_{m\in\Z} \w{m}^{2t}\abs{q_{m}^{\cos}}^{2} }^{1/2} \le C_{t}\n{q}_{t,2}.
\end{equation}

\emph{Schrödinger operators with singular potentials.}
For any $q\in H_{0,\C}^{-1}$ denote by $L(q)$ the unbounded operator $-\partial_{x}^{2} + V$ acting on $H^{-1}(\R/2\Z,\C)$ with domain $H^{1}(\R/2\Z,\C)$. As $H^{1}(\R/2\Z,\C) = H^{1}_{\per+}\oplus H^{1}_{\per-}$ and $V = V_{\per+}\oplus V_{\per-}$ the operator $L(q)$ leaves the spaces $H_{\per\pm}^{1}$ invariant and $L(q) = L_{\per+}(q)\oplus L_{\per-}(q)$ with $L_{\per\pm}(q) = -\partial_{x}^{2} + V_{\per\pm}$. Hence the spectrum $\spec(L(q))$ of $L(q)$, also referred to as spectrum of $q$, is the union $\spec(L_{\per+}(q))\cup \spec(L_{\per-}(q))$ of the spectra $\spec(L_{\per\pm}(q))$ of $L_{\per\pm}(q)$. 
The spectrum $\spec(L(q))$ is known to be discrete and to consist of complex eigenvalues which, when counted with multiplicities and ordered lexicographically, satisfy
\[
  \lm_{0}^{+} \lex \lm_{1}^{-} \lex \lm_{1}^{+} \lex \dotsb,\qquad
  \lm_{n}^{\pm} = n^{2}\pi^{2} + n\ell_{n}^{2},
\]
-- see e.g. \cite{Kappeler:2003vh}. For any $q\in H_{0,\C}^{-1}$ there exists $N\ge 1$ so that
\begin{equation}
  \label{per-rough-estimate}
  \abs{\lm_{n}^{\pm}-n^{2}\pi^{2}} \le n/2,\quad n\ge N,\qquad
  \abs{\lm_{n}^{\pm}} \le (N-1)^{2}\pi^{2}+N/2,\quad n < N,
\end{equation}
where $N$ can be chosen locally uniformly in $q$ on $H_{0,\C}^{-1}$. Since for $q=0$ and $n\ge 0$, $\Dl(\lm_{2n}^{+}(0),0) = 2$ and $\Dl(\lm_{2n+1}^{+}(0),0) = -2$, all $\lm_{2n}^{+}(0)$ are 1-periodic and all $\lm_{2n+1}^{+}(0)$ are 1-antiperiodic eigenvalues of $q = 0$. By considering the compact interval $[0,q] = \setdef{tq}{0\le t\le 1}\subset H_{0,\C}^{-1}$ it then follows after increasing $N$, if necessary, that for any $n\ge N$
\begin{equation}
  \label{ev-per-antiper}
  \lm_{n}^{+}(q),\lm_{n}^{-}(q) \in \spec(L_{\per+}(q)),\;[\spec(L_{\per-}(q))]
  \quad \text{if }n\text{ even [odd]}.
\end{equation}

For any $q\in H_{0,\C}^{-1}$ and $n\ge N$ the following Riesz projectors are thus well defined on $H^{-1}(\R/2\Z,\C)$
\[
  P_{n,q} \defl \frac{1}{2\pi\ii}\int_{\abs{\lm - n^{2}\pi^{2 }} = n} (\lm-L(q))^{-1}\,\dlm.
\]
For $n\ge N$ even [odd], the range of $P_{n,q}$ is contained in $H_{\per+}^{1}$ [$H_{\per-}^{1}$]. For $q = 0$, $P_{n,0}$ coincides with the projector $P_{n}$ introduced in~\eqref{Pn-Qn}.

Similarly, for any $q\in H_{0,\C}^{-1}$ denote by $L_{\dir}(q)$ the unbounded operator $-\partial_{x}^{2}+ V_{\dir}$ acting on $H_{\dir}^{-1}$ with domain $H_{\dir}^{1}$. Its spectrum $\spec(L_{\dir}(q))$ is known to be discrete and to consist of complex eigenvalues which, when counted with multiplicities and ordered lexicographically, satisfy
\[
  \mu_{1}\lex \mu_{2}\lex \dotsb,\qquad \mu_{n} = n^{2}\pi^{2} + n\ell_{n}^{2},
\]
-- see e.g. \cite{Kappeler:2003vh}. By increasing the number $N$ chosen above, if necessary, we can thus assume that
\begin{equation}
  \label{dir-rough-estimate}
  \abs{\mu_{n}-n^{2}\pi^{2}} < n/2,\quad n\ge N,\qquad
  \abs{\mu_{n}} \le (N-1)^{2}\pi^{2} + N/2,\quad n\le N.
\end{equation}
In particular, for any $n\ge N$, $\mu_{n}$ is simple and the corresponding Riesz projector
\[
  \Pi_{n,q} \defl \frac{1}{2\pi\ii} \int_{\abs{\lm-n^{2}\pi^{2}}=n} (\lm-L_{\dir}(q))^{-1}\,\dlm
\]
is well defined on $H_{\dir}^{-1}$. If $q=0$, we write $\Pi_{n}$ for $\Pi_{n,0}$.

\smallskip

\emph{Regularity of solutions.}
In Section~\ref{s:spectral-theory} we consider solutions $f$ of the equation $(L(q)-\lm)f = g$ in $\FL_{\star,\C}^{s,\infty}$ and need to know their regularity.

\begin{lem}
\label{regularity-inhomogeneous}
For any $q\in \FL_{0,\C}^{s,\infty}$ with $-1/2 < s \le 0$, the following holds: For any $g\in \FL_{\star,\C}^{s,\infty}$ and any $\lm\in\C$, a solution $f\in H^{1}(\R/2\Z,\C)$ of the inhomogeneous equation $(L(q)-\lm)f = g$ is an element in $\FL_{\star,\C}^{s+2,\infty}$.~\fish
\end{lem}

\begin{proof}
Let $g\in \FL_{\star,\C}^{s,\infty}$ and assume that $f\in H^{1}(\R/2\Z,\C)$ solves $(L(q)-\lm)f = g$. Write $(L-\lm)f = g$ as $A_{\lm}f = Vf - g$ where $A_{\lm} = \partial_{x}^{2} + \lm$. Since $q\in \FL_{0,\C}^{s,\infty}$, Lemma~\ref{ell-embedding} implies that $q$ and $g$ are in $\FL_{\star,\C}^{r,2}$ with $r = s-1/2-\ep$ where $\ep > 0$ is chosen such that $r > -1$. By Lemma~\ref{convolution} (i), $Vf\in \FL_{\star,\C}^{r,2}$ and hence $A_{\lm}f = Vf-g\in\FL_{\star,\C}^{r,2}$ implying that $f\in \FL_{\star,\C}^{r+2,2}$. Since $-s-3/2 \le -1 < r \le 0$, Lemma~\ref{convolution} (ii) applies. Therefore, $Vf\in \FL_{\star,\C}^{s,\infty}$ and using the equation $A_{\lm}f = Vf - g$ once more one gets $f\in \FL_{\star,\C}^{s+2,\infty}$ as claimed.~\qed
\end{proof}

For any $q\in \FL_{0,\C}^{s,\infty}$ with $-1/2 < s \le 0$, and $n\ge n_{s}$ as in Corollary~\ref{Q-soln}, introduce
\[
  E_{n} \equiv E_{n}(q) \defl 
  \begin{cases}
  \Null(L(q)-\lm_{n}^{+})\oplus \Null(L(q)-\lm_{n}^{-}), & \lm_{n}^{+}\neq \lm_{n}^{-},\\
  \Null(L(q)-\lm_{n}^{+})^{2}, & \lm_{n}^{+} = \lm_{n}^{-}.
  \end{cases}
\]
Then $E_{n}$ is a two-dimensional subspace of $H_{\per+}^{1}$ [$H_{\per-}^{1}$] if $n$ is even [odd]. The following result shows that elements in $E_{n}$ are more regular.

\begin{lem}
\label{regularity-En}
For any $q\in \FL_{0,\C}^{s,\infty}$ with $-1/2 < s \le 0$ and for any $n\ge n_{s}$ $E_{n}(q)\subset\FL_{\star,\C}^{s+2,\infty}\cap H_{\per+}^{1}$ [$\FL_{\star,\C}^{s+2,\infty}\cap H_{\per-}^{1}$] if $n$ is even [odd].~\fish
\end{lem}

\begin{proof}
By Lemma~\ref{regularity-inhomogeneous} with $g=0$, any eigenfunction $f$ of an eigenvalue $\lm$ of $L(q)$ is in $\FL_{\star,\C}^{s+2,\infty}$. Hence if $\lm_{n}^{+} \neq \lm_{n}^{-}$ or if $\lm_{n}^{+} = \lm_{n}^{-}$ and has geometric multiplicity two, then $E_{n}\subset\FL_{\star,\C}^{s+2,\infty}$. Finally, if $\lm_{n}^{+} = \lm_{n}^{-}$ is a double eigenvalue of geometric multiplicity $1$ and $g$ is an eigenfunction corresponding to $\lm_{n}^{+}$, there exists an element $f\in H^{1}(\R/2\Z,\C)$ so that $(L-\lm_{n}^{+})f = g$. Since $g$ is an eigenfunction it is in $\FL_{\star,\C}^{s+2,\infty}$ by Lemma~\ref{regularity-inhomogeneous} and by applying this lemma once more, it follows that $f\in \FL_{\star,\C}^{s+2,\infty}$. Clearly, $E_{n} = \mathrm{span}(g,f)$ and hence $E_{n}\subset \FL_{\star,\C}^{s+2,\infty}$ also in this case. By~\eqref{ev-per-antiper}, $\lm_{n}^{\pm}$ are 1-periodic [1-antiperiodic] eigenvalues of $q$ if $n$ is even [odd]. Hence $E_{n}\subset \FL_{\star,\C}^{s+2,\infty}\cap H_{\per+}^{1}$ [$\FL_{\star,\C}^{s+2,\infty}\cap H_{\per-}^{1}$] if $n$ is even [odd] as claimed.~\qed
\end{proof}

\emph{Estimates for projectors.}
The projectors $P_{n,q}$ [$\Pi_{n,q}$] with $n\ge N$ and $N$ given by \eqref{per-rough-estimate}-\eqref{dir-rough-estimate} are defined on $H^{-1}(\R/2\Z,\C)$ [$H_{\dir}^{-1}$] and have range in $H^{1}(\R/2\Z,\C)$ [$H^{1}_{\dir}$]. The following result concerns estimates for the restriction of $P_{n,q}$ [$\Pi_{n,q}$] to $L^{2} = L^{2}([0,2],\C)$ [$L^{2}(\Ic) = L^{2}([0,1],\C)$] needed in Section~\ref{s:spectral-theory} for getting the asymptotics of $\mu_{n}-\tau_{n}$ stated in Theorem~\ref{thm:spec-fw} (ii).

\begin{lem}
\label{proj-bound}
Assume that $q\in H_{0,\C}^{-1}$ with $\n{q}_{t,2} < \infty$ and $-1 < t < -1/2$. Then there exist constants $C_{t} > 0$ (only depending on $t$) and $N'\ge N$ (with $N$ as above) so that for any $n\ge N$ the following holds
\begin{enumerate}[label=(\roman{*})]
\item
\[
  \n{P_{n,q}-P_{n}}_{L^{2}\to L^{\infty}}
  \le C_{t} \frac{(\log n)^{2}}{n^{1-\abs{t}}}\n{q}_{t,2}
\]
\item
\[
  \n{\Pi_{n,q}-\Pi_{n}}_{L^{2}(\Ic)\to L^{\infty}(\Ic)}
  \le C_{t} \frac{(\log n)^{2}}{n^{1-\abs{t}}}\n{q}_{t,2}
\]
\end{enumerate}
The constant $N'$ can be chosen locally uniformly in $q$.~\fish
\end{lem}

\begin{proof}
\cite[Lemma~25]{Kappeler:CNzeErmy}.~\qed
\end{proof}

\begin{rem}
We will apply Lemma~\ref{proj-bound} for potentials $q\in\FL_{0,\C}^{s,\infty}$ with $-1/2 < s\le 0$  using the fact that by the Sobolev embedding theorem there exists $-1 < t < -1/2$ so that $\FL_{0,\C}^{s,\infty} \opento H_{0,\C}^{t}$.~\map
\end{rem}

\end{document}